\documentclass[draft]{amsart}

\usepackage{amsmath,amssymb,amsthm,enumerate} 
\usepackage{amsfonts} 
\usepackage[dvipdfmx]{graphicx}
\usepackage{color}

\theoremstyle{plain}
\newtheorem{pretheo}{Theorem}[section]
\newtheorem{preassu}[pretheo]{Assumption}
\newtheorem{precoro}[pretheo]{Corollary}
\newtheorem{predefi}[pretheo]{Definition}
\newtheorem{preexam}[pretheo]{Example}
\newtheorem{prelemm}[pretheo]{Lemma}
\newtheorem{preprop}[pretheo]{Proposition}
\newtheorem{prerema}[pretheo]{Remark}

\newenvironment{theo}{\begin{pretheo}}{\end{pretheo}}

\newenvironment{coro}{\begin{precoro}}{\end{precoro}}
\newenvironment{defi}{\begin{predefi}}{\end{predefi}}

\newenvironment{lemm}{\begin{prelemm}}{\end{prelemm}}
\newenvironment{prop}{\begin{preprop}}{\end{preprop}}
\newenvironment{rema}{\begin{prerema}\rm}{\end{prerema}}

\DeclareMathOperator{\di}{div}

\DeclareMathOperator{\Di}{Div}

\newcommand{\intd}{\,d}

\newcommand{\loc}{{\rm loc}}
\newcommand{\pa}{\partial}

\newcommand{\wh}[1]{\widehat{#1}}
\newcommand{\wt}[1]{\widetilde{#1}}

\newcommand{\Ba}{\mathbf{a}}
\newcommand{\Bb}{\mathbf{b}}

\newcommand{\Be}{\mathbf{e}}
\newcommand{\Bf}{\mathbf{f}}
\newcommand{\Bg}{\mathbf{g}}
\newcommand{\Bh}{\mathbf{h}}

\newcommand{\Bn}{\mathbf{n}}

\newcommand{\Bu}{\mathbf{u}}
\newcommand{\Bv}{\mathbf{v}}

\newcommand{\By}{\mathbf{y}}

\newcommand{\BA}{\mathbf{A}}
\newcommand{\BB}{\mathbf{B}}
\newcommand{\BC}{\mathbf{C}}
\newcommand{\BD}{\mathbf{D}}

\newcommand{\BF}{\mathbf{F}}
\newcommand{\BG}{\mathbf{G}}
\newcommand{\BH}{\mathbf{H}}
\newcommand{\BI}{\mathbf{I}}
\newcommand{\BJ}{\mathbf{J}}

\newcommand{\BM}{\mathbf{M}}
\newcommand{\BN}{\mathbf{N}}
\newcommand{\BO}{\mathbf{O}}

\newcommand{\BR}{\mathbf{R}}

\newcommand{\BT}{\mathbf{T}}
\newcommand{\BU}{\mathbf{U}}
\newcommand{\BV}{\mathbf{V}}
\newcommand{\BW}{\mathbf{W}}
\newcommand{\BX}{\mathbf{X}}

\newcommand{\CA}{\mathcal{A}}
\newcommand{\CB}{\mathcal{B}}

\newcommand{\CF}{\mathcal{F}}
\newcommand{\CG}{\mathcal{G}}

\newcommand{\CL}{\mathcal{L}}

\newcommand{\CN}{\mathcal{N}}

\newcommand{\CV}{\mathcal{V}}
\newcommand{\CW}{\mathcal{W}}

\newcommand{\Fg}{\mathfrak{g}}

\newcommand{\Fp}{\mathfrak{p}}
\newcommand{\Fq}{\mathfrak{q}}

\newcommand{\SST}{\mathsf{T}}

\newcommand{\tws}{{\BR^{N-1}}}
\newcommand{\ws}{{\BR^N}}

\newcommand{\al}{\alpha}

\newcommand{\ga}{\gamma}
\newcommand{\de}{\delta}
\newcommand{\ep}{\varepsilon}
\newcommand{\te}{\theta}

\newcommand{\la}{\lambda}
\newcommand{\si}{\sigma}
\newcommand{\ph}{\varphi}
\newcommand{\om}{\omega}

\newcommand{\Ga}{\Gamma}
\newcommand{\De}{\Delta}
\newcommand{\Te}{\Theta}

\newcommand{\Si}{\Sigma}
\newcommand{\Om}{\Omega}

\numberwithin{equation}{section} 
\allowdisplaybreaks[4]

\begin{document}

\title[The Navier-Stokes equations with a free surface]
{Global solvability of the Navier-Stokes equations
with a free surface in the maximal $L_p\text{-}L_q$ regularity class}

\author[Hirokazu Saito]{Hirokazu Saito}

\address{Department of Mathematics, Faculty of Science and Engineering, 
Waseda University, Okubo 3-4-1, Shinjuku-ku, Tokyo 169-8555, Japan}

\email{hsaito@aoni.waseda.jp}







\begin{abstract}
We consider the motion of incompressible viscous fluids
bounded above by a free surface and below by a solid surface
in the $N$-dimensional Euclidean space for $N\geq 2$.
The aim of this paper is to show the global solvability
of the Naiver-Stokes equations with a free surface,
describing the above-mentioned motion,
in the maximal $L_p\text{-}L_q$ regularity class.
Our approach is based on the maximal $L_p\text{-}L_q$ regularity
with exponential stability for the linearized equations,
and also it is proved that solutions to the original nonlinear problem are exponentially stable.
\end{abstract}

\maketitle

\renewcommand{\thefootnote}{\arabic{footnote})}


\section{Introduction}\label{sec:intro}
This paper is concerned with the global solvability of the Navier-Stokes equations with a free surface,
describing the motion of incompressible viscous fluids
bounded above by a free surface and below by a solid surface
in the $N$-dimensional Euclidean space for $N\geq 2$,
in the maximal $L_p\text{-}L_q$ regularity class (cf. \cite{Shibata15} for the class).
Such equations were mathematically treated by Beale \cite{Beale81} for the first time.
He proved, in an $L_2$-in-time and $L_2$-in-space setting with the gravity,
the local solvability for large initial data in \cite{Beale81},
whereas we prove in the maximal $L_p\text{-}L_q$ regularity class
the global solvability for small initial data
in the case where the gravity is not taken into account in the present paper.

The problem is stated as follows:
We are given an initial domain $\Om\subset\ws$,
occupied by an incompressible viscous fluid, such that
\begin{equation*}
\Om = \{\xi=(\xi',\xi_N) \mid \xi'=(\xi_1,\dots,\xi_{N-1})\in\tws,0<\xi_N<d\} \quad (d>0),
\end{equation*}
as well as an initial velocity field
$\Ba=\Ba(\xi)=(a_1(\xi),\dots,a_N(\xi))^\SST\footnote{$\BM^\SST$ denotes the transposed $\BM$.}$
of the fluid on $\Om$.
The symbols $\Ga$, $S$ denote the boundaries of $\Om$ such that
\begin{align*}
\Ga &= \{\xi=(\xi',\xi_N) \mid \xi'=(\xi_1,\dots,\xi_{N-1})\in\tws,\xi_N=d\}, \\
S&= \{\xi=(\xi',\xi_N) \mid \xi'=(\xi_1,\dots,\xi_{N-1})\in\tws,\xi_N=0\}.
\end{align*}
We wish to find for each $t\in(0,\infty)$
a transformation $\Te=\Te(\cdot,t):\Om\ni\xi\mapsto x=\Te(\xi,t)\in\ws$,
a velocity field $\Bv = \Bv(x,t)=(v_1(x,t),\dots,v_N(x,t))^\SST$ of the fluid,
and a pressure field $\pi = \pi(x,t)$ of the fluid so that
\begin{eqnarray}
&\pa_t\Te = \Bv\circ\Te, \quad \Te(\xi,0)=\xi,  \quad\text{$\xi\in\Om$,} \label{eq:2}  \\
&\Om(t) = \Te(\Om,t), \quad \Ga(t) = \Te(\Ga,t), \quad S = \Te(S,t), \label{eq:1} \\
&\pa_t\Bv+(\Bv\cdot\nabla)\Bv = \Di\BT(\Bv,\pi), \quad \text{$x\in\Om(t)$,} \label{eq:3} \\
&\di\Bv = 0, \quad \text{$x\in \Om(t)$,} \label{eq:4} \\
&\BT(\Bv,\pi)\Bn = -\pi_0\Bn, \quad \text{$x\in\Ga(t)$,} \label{eq:5} \\
&\Bv = 0, \quad \text{$x\in S$,} \label{eq:6} \\
&\Bv|_{t=0} = \Ba, \quad \text{$\xi\in\Om$,} \label{eq:7}
\end{eqnarray}
where $\Bv\circ\Te = (\Bv\circ\Te)(\xi,t)=\Bv(\Te(\xi,t),t)$.

Here the density of the fluid have been set to $1$;
$\Bn$ is the unit outward normal to $\Ga(t)$;
the constant $\pi_0$ is the atmospheric pressure,
and it is assumed in this paper that $\pi_0=0$ without loss of generality.
The stress tensor $\BT(\Bv,\pi)$ is given by $\BT(\Bv,\pi)=\mu\BD(\Bv)-\pi\BI$,
where $\mu$ is a positive constant and denotes the viscosity coefficient of the fluid;
$\BI$ is the $N \times N$ identity matrix;
$\BD(\Bv) = \nabla\Bv+(\nabla\Bv)^\SST$ is the doubled strain tensor.
Here and subsequently, we use the following notation for differentiations:
Let $f=f(x)$, $\Bg=(g_1(x),\dots,g_N(x))^\SST$, and $\BM=(M_{i j}(x))$
be a scalar-, a vector-, and an $N\times N$ matrix-valued function on a domain of $\ws$,
respectively, and then for $\pa_j = \pa/\pa x_j$
\begin{align*}
&\nabla f = (\pa_1 f,\dots,\pa_N f)^\SST, \quad
\De f = \sum_{j=1}^N\pa_j^2 f, \quad
\De\Bg = (\De g_1,\dots,\De g_N)^\SST, \\
&\di\Bg = \sum_{j=1}^N\pa_j g_j, \quad
\nabla^2\Bg = \{\pa_i\pa_j g_k \mid i,j,k=1,\dots,N\}, \\
&\nabla\Bg =
\begin{pmatrix}
\pa_1 g_1 & \dots & \pa_N g_1 \\
\vdots & \ddots & \vdots \\
\pa_1 g_N & \dots & \pa_N g_N
\end{pmatrix}, \quad
(\Bg\cdot\nabla)\Bg = \left(\sum_{j=1}^N g_j\pa_j g_1,\dots,\sum_{j=1}^N g_j\pa_j g_N\right)^\SST, \\
&\Di\BM = \left(\sum_{j=1}^N\pa_j M_{1 j},\dots,\sum_{j=1}^N \pa_j M_{N j}\right)^\SST. 
\end{align*}

Let $\Bu(\xi,t)=(\Bv\circ\Te)(\xi,t)$, which is the so-called {\it Lagrangian velocity},
and then the solution $\Te$ to \eqref{eq:2} is represented as
\begin{equation}\label{trans}
\Te(\xi,t)=\xi + \int_0^t\Bu(\xi,s)\intd s \quad
\text{($\xi\in\Om$, $t>0$).}
\end{equation}

We now write the equations \eqref{eq:3}-\eqref{eq:7} in the Lagrangian formulation.
Thus our unknowns will be the Lagrangian velocity $\Bu(\xi,t)=\Bv(\Te(\xi,t),t)$
and pressure $\Fp(\xi,t)=\pi(\Te(\xi,t),t)$ for $(\xi,t)\in\Om\times(0,\infty)$.
If we substitute the new unknowns $\Bu=\Bu(\xi,t)$
and $\Fp=\Fp(\xi,t)$ in \eqref{eq:3}-\eqref{eq:7}, then the equations
turn into\footnote{The derivation of \eqref{eq:11}-\eqref{eq:13} is discussed in the appendix.}
\begin{eqnarray}
\pa_t\Bu-\Di\BT(\Bu,\Fp) = \BF(\Bu) && \text{in $\Om$, $t>0$,} \label{eq:11} \\
\di\Bu = G(\Bu) = \di\BG(\Bu) && \text{in $\Om$, $t>0$,} \label{eq:12} \\
\BT(\Bu,\Fp)\Be_N = \BH(\Bu)\Be_N && \text{on $\Ga$, $t>0$,} \label{eq:13} \\
\Bu = 0 && \text{on $S$, $t>0$,} \label{eq:14} \\
\Bu|_{t=0} = \Ba && \text{in $\Om$,} \label{eq:15}
\end{eqnarray}
where $\Be_N=(0,\dots,0,1)^\SST$.
Here $\BF(\Bu)$, $G(\Bu)$, $\BG(\Bu)$, and $\BH(\Bu)$ are nonlinear terms,
with respect to $\Bu$, of the forms:
\begin{align}\label{nonlinear:1}
&\BF(\Bu) =
\BU_1\left(\int_0^t\nabla\Bu(\xi,s)\intd s\right)\pa_t\Bu
+ \BV\left(\int_0^t\nabla\Bu(\xi,s)\intd s\right)\nabla^2 \Bu \\
&\quad +\left[\BW\left(\int_0^t\nabla\Bu(\xi,s)\intd s\right)
\int_0^t\nabla^2\Bu(\xi,s)\intd s\right]\nabla\Bu, \notag \\
&G(\Bu) =
\BU_2\left(\int_0^t\nabla\Bu(\xi,s)\intd s\right):\nabla\Bu, \quad
\BG(\Bu) =
\BU_3\left(\int_0^t\nabla\Bu(\xi,s)\intd s\right)\Bu, \notag \\
&\BH(\Bu) =
\BD(\Bu)\BU_4\left(\int_0^t\nabla\Bu(\xi,s)\intd s\right)
+ (\nabla\Bu)^\SST\BU_5\left(\int_0^t\nabla\Bu(\xi,s)\intd s\right) \notag \\
&\quad +\BU_6\left(\int_0^t\nabla\Bu(\xi,s)\intd s\right)
(\nabla\Bu)^\SST\left(\BI+\BU_7\left(\int_0^t\nabla\Bu(\xi,s)\intd s\right)\right), \notag
\end{align}
where $\BU_i:\BR^{N\times N}\to \BR^{N\times N}$ $(i=1,\dots,7)$, 
\begin{align*}
\BV(\cdot)\nabla^2\Bu
=&\Bigg(\sum_{i,j,k=1}^NV_{ijk}^1(\cdot)\pa_i\pa_j u_k,\dots,\sum_{i,j,k=1}^N V_{ijk}^N(\cdot)\pa_i\pa_j u_k\Bigg)^\SST, \\
\left[\BW(\cdot)\int_0^t\nabla^2\Bu\intd s\right]\nabla\Bv
=&\Bigg(\sum_{i,j,k,l,m=1}^N W_{ijklm}^1(\cdot)\int_0^t\pa_i\pa_j u_k \intd s\,\pa_l v_m, \\
&\quad\dots,\sum_{i,j,k,l,m=1}^N W_{ijklm}^N(\cdot)\int_0^t\pa_i\pa_j u_k \intd s\,\pa_l v_m\Bigg)^\SST,
\end{align*} 
for some $V_{ijk}^1,\dots, V_{ijk}^N, W_{ijklm}^1,\dots, W_{ijklm}^N:\BR^{N\times N}\to\BR$
and for any $N$-vectors $\Bu=(u_1,\dots,u_N)^\SST$, $\Bv=(v_1,\dots,v_N)^\SST$.
Note that, for $N \times N$ matrices $\BA=(A_{i j})$, $\BB=(B_{i j})$, we have set
\begin{equation*}
\BA:\BB=\sum_{i,j=1}^N A_{i j}B_{i j}.
\end{equation*}
One has the following information about $\BU_i$, $\BV$, $\BW$:
Let $\BX=(X_{mn})$ be $N\times N$ matrices.
Then all the components of $\BU_i(\BX)$ $(i=1,\dots,7)$, $\BV(\BX)$, and $\BW(\BX)$
are polynomials with respect to $X_{mn}$ for $m,n=1,\dots,N$.
Furthermore,
\begin{equation}\label{null}
\BU_i(\BO)=\BO \quad (i=1,\dots,7),  \quad
V_{ijk}^l(\BO) =0 \quad (i,j,k,l=1,\dots,N),  
\end{equation}
where $\BO$ denotes the $N\times N$ zero matrix. 

Let us introduce historical remarks and key ideas of the present paper at this point.
If we consider free boundary problems,  
then we first usually transform them to nonlinear problems on given domains,
independent of time $t$, by using a suitable transformation.
Roughly speaking, such transformations  are divided into

\begin{itemize}
\item Lagrangian transformation;
\item Eulerian approaches (e.g. Beale's transformation in \cite{Beale83}, \cite{BN85}; Hanzawa's transformation in \cite{Hanzawa81}).
\end{itemize}

{\it Lagrangian transformation} denotes the transformation $\Te$ of \eqref{trans}.
It is quite useful to show the local solvability for a lot of situations.
In fact, by using the Lagrangian transformation,
Shibata \cite{Shibata15} and Enomoto et al. \cite{EBS14} proved, respectively, the local solvability
of the Navier-Stokes eqautions with a free surface
for incompressible viscous fluids and for compressible viscous fluids 
in the case where the initial domain $\Om$ has uniform $W_q^{2-1/q}$ regularity (cf. their papers for the definition).
Here half-spaces, bent half-spaces, layer-like domains, cylinder-like domains, bounded domains,
and exterior domains are typical examples of uniform $W_q^{2-1/q}$ domains.
As for the local solvability with the Lagrangian transformation, we also refer e.g. to the following papers:
Solonnikov \cite{Solonnikov03, Solonnikov03b} and references therein,
Mogilevskii and Solonnikov \cite{MS92}, Mucha and Zaj\c{a}czkowski \cite{MZ00}, Shibata and Shimizu \cite{SS07}
for incompressible viscous fluids in bounded domains;
Beale \cite{Beale81}, Allain \cite{Allain87}, Tani \cite{Tani96}, Abels \cite{Abels05}
for incompressible viscous fluids in layer-like domains;
Tani \cite{Tani81}, Solonnikov and Tani \cite{ST90, ST91},
Secchi and Valli \cite{SV83}, Secchi \cite{Secchi90b, Secchi90, Secchi91},
Zaj\c{a}czkowski \cite{Zajaczkowski92, Zajaczkowski93, Zajaczkowski95},
Zadrzy\'nska and Zaj\c{a}czkowski \cite{ZZ94, ZZ98, ZZ00},
Str\"ohmer and Zaj\c{a}czkowski \cite{SZ99}, Denisova and Solonnikov \cite{DS02}
for compressible viscous fluids in bounded domains;
Tanaka and Tani \cite{TT03} for compressible viscous fluids in layer-like domains;
Tani \cite{Tani84}, 	Denisova \cite{Denisova94, Denisova97, Denisova00, Denisova01, Denisova03},
Denisova and Solonnikov \cite{DS95}
for the motion of two fluids separated by a closed free surface.

%
%




The advantage of the Lagrangian transformation for the local solvability is that
$\int_0^t\nabla\Bu(\xi,s)\intd s$ appears in nonlinear terms (cf. \eqref{nonlinear:1}).
By choosing $T>0$ small enough, we can see $\int_0^t\nabla\Bu(\xi,s)\intd s$ $(t\in(0,T))$ as a small coefficient,
and thus the nonlinear terms would be small with suitable norms.
This enables us to show the local solvability by using the contraction mapping theorem.

In order to prove the global solvability,
we need the integrability of $\nabla\Bu(\xi,t)$ with respect to time $t\in(0,\infty)$
because of $\int_0^t\nabla\Bu(\xi,s)\intd s$. 
If we are in an $L_p$-in-time and $L_q$-in-space setting,
then a key idea to guarantee the time integrability is exponential stability of solutions for the linearized equations
as was seen e.g. in Shibata-Shimizu \cite{SS08}, Shibata \cite{Shibata15}.
These two papers tell us that the exponential stability can be proved for bounded domains
by some abstract approach.
However, for unbounded domains containing $\Om$ of the present paper, it is not true in general that the exponential stability holds. 
This is one of main difficulties to prove the global solvability in our situation.

{\it Eulerian approaches} are useful to show the large-time behavior of solutions to
the Navier-Stokes equations with a free surface, 
because $\int_0^t \Bu(\xi,s)\intd s$ as above does not appear.
We here introduce e.g. the following papers as references of the Eulerian approaches:
Beale and Nishida \cite{BN85} (cf. also \cite{Hataya11} for the detailed proof), Hataya and Kawashima \cite{HK09}
proved polynomial decay of solutions for layer-like domains in an $L_2$-in-time and $L_2$-in-space setting;
K\"ohne, Pr\"uss, and Wilke \cite{KPW13}, Solonnikov \cite{Solonnikov14} proved exponential stability of solutions for bounded domains 
in an $L_p$-in-time and $L_p$-in-space setting;
Saito and Shibata \cite{SaS16} showed $L_q\text{-}L_r$ estimates of the Stokes semigroup
with surface tension and gravity on the half-space.

The key idea of this paper is to prove 
the maximal $L_p\text{-}L_q$ regularity with exponential stability 
for the linearized equations of \eqref{eq:11}-\eqref{eq:15}.
As mentioned above, we can prove such an exponential stability for bounded domains,
but the technique can not be applied to our situation because our domain $\Om$ is unbounded.
To overcome this difficulty,  we make use of Abels \cite{Abels06} and Saito \cite{Saito15b} in this paper.
We then apply the maximal $L_p\text{-}L_q$ regularity with exponential stability 
and the contraction mapping theorem to  \eqref{eq:11}-\eqref{eq:15}
in order to prove their global solvability. 

This paper is organized as follows:
The next section first introduce notation and function spaces
that are used throughout this paper.
Secondly, we state main results of the global solvability for \eqref{eq:11}-\eqref{eq:15}
and the maximal $L_p\text{-}L_q$ regularity with exponential stability for the linearized equations.
Section 3 introduces some results concerning
a time-shifted problem,  a variational problem, and the Helmholtz decomposition on $\Om$.
Section 4 shows the generation of an analytic $C_0$-semigroup 
associated with the linearized equations.
In Section 5, we prove the maximal $L_p\text{-}L_q$ regularity with exponential stability
for the linearized equations by means of results introduced in Sections 3, 4.
Section 6 proves the global solvability of \eqref{eq:11}-\eqref{eq:15}
based on the contraction mapping theorem together with
the maximal $L_p\text{-}L_q$ regularity with exponential stability proved in Section 5.
In Section 7, we show the global existence and uniqueness of solutions for the original problem \eqref{eq:2}-\eqref{eq:7},
and also their exponential stability.

\section{Notation and main results}\label{sec:main}
In this section, we first introduce notation and function spaces
that are used throughout this paper.
Next our main results are stated.

\subsection{Notation}
The set of all natural numbers, real numbers, and complex numbers are denoted by $\BN$, $\BR$, and $\BC$, respectively,
and $\BN_0 = \BN \cup \{0\}$.

Let $m,n\in\BN$ and $G$ be a domain of $\BR^n$.
We set $\Ba\cdot\Bb=\sum_{j=1}^m a_j b_j$
for $m$-vectors $\Ba=(a_1,\dots,a_m)^\SST$ and $\Bb=(b_1,\dots,b_m)^\SST$,
while we set $(\Bf,\Bg)_G=\int_G \Bf(x)\cdot\Bg(x) \intd x = \sum_{j=1}^m\int_G f_j(x)g_j(x) \intd x$
for $m$-vector functions $\Bf=\Bf(x)=(f_1(x),\dots,f_m(x))^\SST$, $\Bg=\Bg(x)=(g_1(x),\dots,g_m(x))^\SST$ on $G$.
In addition,  $C_0^\infty(G)$ denotes the set of all $C^\infty$-functions on $\BR^n$
whose supports are compact and contained in $G$.

Let $X$ and $Y$ be Banach spaces, and let $1\leq p \leq \infty$.
The Banach space of all bounded linear operators from $X$ to $Y$ is denoted by $\CL(X,Y)$, and $\CL(X)=\CL(X,X)$. 
$L_p(G,X)$ and $W_p^m(G,X)$ denote, respectively, the standard $X$-valued Lebesgue spaces on $G$
and the standard $X$-valued Sobolev spaces on $G$, and $W_p^0(G,X)=L_p(G,X)$.
If $X=\BR$ or $X=\BC$, then $L_p(G,X)$, $W_p^m(G,X)$, and $W_p^0(G,X)$ are denoted by
$L_p(G)$, $W_p^m(G)$, and $W_p^0(G)$, respectively, for short.

The symbol $C(G,X)$ stands for the set of all $X$-valued continuous function on $G$,
while $C^m(G,X)$ is the set of all $m$-times continuously differentiable functions on $G$ with values in $X$.
Let $BUC(G,X)$ be the Banach space of all $X$-valued uniformly continuous and bounded functions on $G$.
In addition, 
\begin{equation*}
BUC^m(G,X)=\{f\in C^m(G,X) \mid \pa^\al f\in BUC(G,X) \text{ for $|\al|=0,1,\dots,m$}\},
\end{equation*}
where we have set $\pa^\al f = \pa_1^{\al_1}\dots\pa_n^{\al_n}f$ and $|\al|=\al_1+\dots+\al_n$
for multi-index $\al=(\al_1,\dots,\al_n)\in\BN_0^n$.
Here $C(G)$, $C^m(G)$, $BUC(G)$, and $BUC^m(G)$ are defined similarly as above,
and also $G$ can be replaced by the closure $\overline{G}$ of $G$.

Let $\ga\in\BR$. We then define functions spaces with exponential weights as 
\begin{align*}
L_{p,\ga}(\BR,X) &= \{f \in L_{p,\loc}(\BR,X) \mid e^{-\ga t}f(t) \in L_p(\BR,X)\}, \\
W_{p,\ga}^1(\BR,X) &= \{f \in W_{p,\loc}^1(\BR,X) \mid e^{-\ga t}\pa_t^k f(t)\in L_p(\BR,X) \text{ for $k=0,1$}\}.
\end{align*}
On the other hand, one sets for $\BR_+=(0,\infty)$
\begin{equation*}
{}_0 W_p^1(\BR_+,X) = \{f\in W_p^1(\BR_+,X) \mid f|_{t=0}=0\}.
\end{equation*}

In order to define Bessel potential spaces of order $1/2$, 
we introduce the Fourier transform and
its inverse transform as follows:
Let $f=f(t)$ and $g=g(\tau)$ be functions defined on $\BR$, 
and then
\begin{align*}
\CF[f](\tau) = \int_{\BR}e^{-it\tau}f(t)\intd t, \quad
\CF_{\tau}^{-1}[g](t) = \frac{1}{2\pi}\int_{\BR}e^{it\tau}g(\tau)\intd\tau.
\end{align*} 
The Bessel potential spaces are given by
\begin{align*}
H_p^{1/2}(\BR,X) &= \{f\in L_p(\BR,X) \mid \|f\|_{H_p^{1/2}(\BR,X)}<\infty\}, \\
\|f\|_{H_p^{1/2}(\BR,X)} &=\|\CF_\tau^{-1}[(1+\tau^2)^{1/4}\CF[f](\tau)]\|_{L_p(\BR,X)}, \\
H_{p,\ga}^{1/2}(\BR,X) &= \{f\in L_{p,\ga}(\BR,X) \mid \|e^{-\ga t}f\|_{H_p^{1/2}(\BR,X)}<\infty\}, 
\end{align*}
and furthermore,
\begin{align*}
H_p^{1/2}(\BR_+,X) &=[L_p(\BR_+,X), W_p^1(\BR_+,X)]_{1/2}, \\
{}_0 H_p^{1/2}(\BR_+,X) &= [L_p(\BR_+,X), {}_0 W_p^1(\BR_+,X)]_{1/2},
\end{align*}
where $[\cdot,\cdot]_\te$ is the complex interpolation functor with $0<\te<1$.
For notational convenience, we set for $1<p,q<\infty$ and $Z\in\{H,{}_0 H\}$ 
\begin{align*}
H_{q,p,\ga}^{1,1/2}(\Om\times\BR) &= H_{p,\ga}^{1/2}(\BR,L_q(\Om)) \cap L_{p,\ga}(\BR,W_q^1(\Om)), \\
H_{q,p}^{1,1/2}(\Om\times\BR) &= H_{p}^{1/2}(\BR,L_q(\Om)) \cap L_{p}(\BR,W_q^1(\Om)), \\
Z_{q,p}^{1,1/2}(\Om\times \BR_+) &=  H_p^{1/2}(\BR_+,L_q(\Om))\cap L_p(\BR_+,W_q^1(\Om)),  \\
W_{q,p,\ga}^{2,1}(\Om\times \BR) &= W_{p,\ga}^1(\BR,L_q(\Om)) \cap L_{p,\ga}(\BR,W_q^2(\Om)), \\
W_{q,p}^{2,1}(\Om\times \BR_+) &= W_p^1(\BR_+,L_q(\Om)) \cap L_p(\BR_+,W_q^2(\Om)).
\end{align*}

Let $1<q<\infty$ and $q'=q/(q-1)$.
A closed subspace $W_{q,\Ga}^1(\Om)$ of $W_q^1(\Om)$ is defined as 
$W_{q,\Ga}^1(\Om) = \{f\in W_q^1(\Om) \mid f=0 \text{ on $\Ga$}\}$.
Then the solenoidal space $J_q(\Om)$ is given by 
\begin{equation*}
J_q(\Om) = \{\Bf \in L_q(\Om)^N \mid (\Bf,\nabla\ph)_\Om = 0
\text{ for all $\ph\in W_{q',\Ga}^1(\Om)$}\}.
\end{equation*}
Here we set $D_{q,p}(\Om)=(J_q(\Om),D(A_q))_{1-1/p,p}\footnote{
$D(A_q)$ is the domain of the Stokes operator $A_q$ associated with the linearized equations of \eqref{eq:11}-\eqref{eq:15}.
They are discussed in Section \ref{sec:sg} below in more detail, especially in \eqref{domain}.}$,
where $(\cdot,\cdot)_{\te,p}$ is the real interpolation functor with $0<\te<1$.

\begin{rema}\label{rema:domain}
The interpolation space $D_{q,p}(\Om)$ 
is characterized as follows\footnote{We refer e.g. to \cite[page 4133]{Shibata15}.}: 
\begin{equation*}
D_{q,p}(\Om)
=\left\{\begin{aligned}
&\{\Bu\in J_q(\Om)\cap B_{q,p}^{2-2/p}(\Om)^N \mid (\mu\BD(\Bu)\Be_N)_\tau=0 \text{ on $\Ga$, }\Bu=0 \text{ on $S$}\} \\
&\quad\text{when $2-2/p>1+1/q$,} \\
&\{\Bu\in J_q(\Om)\cap B_{q,p}^{2-2/p}(\Om)^N \mid \Bu=0 \text{ on $S$}\} \\
&\quad \text{when $1/q<2-2/p<1+1/q$,} \\
&J_q(\Om)\cap B_{q,p}^{2-2/p}(\Om)^N \text{ when $2-2/p<1/q$,}
\end{aligned}\right.
\end{equation*}
where we have set $B_{q,p}^{2-2/p}(\Om)=(L_q(\Om),W_q^2(\Om))_{1-1/p,p}$
and $\Bv_\tau=\Bv-\Be_N(\Be_N\cdot\Bv)$ for any $N$-vector $\Bv$.
\end{rema}

Throughout this paper, the letter $C$ denotes generic constants and
$C_{a,b,c,\dots}$ means that the constant depends on the quantities $a,b,c,\dots$.
The values of constants $C$ and $C_{a,b,c,\dots}$ may change from line to line.

\subsection{Main results}\label{subsec:main}
This subsection introduces our main results of this paper.

First the global solvability of \eqref{eq:11}-\eqref{eq:15} is stated as follows:

\begin{theo}\label{theo:main}
Let $p$ and $q$ be exponents satisfying
\begin{equation}\label{pq}
2<p<\infty, \quad N<q<\infty, \quad \frac{2}{p}+\frac{N}{q}<1.
\end{equation}
Then there exist positive constants $\ga_0$, $\de_0$, and $\ep_0$ such that,
for any $\Ba\in D_{q,p}(\Om)$ with $\|\Ba\|_{D_{q,p}(\Om)}\leq \ep_0$,
the equations \eqref{eq:11}-\eqref{eq:15} admit a unique solution
\begin{equation*}
(\Bu,\Fp) \in W_{q,p}^{2,1}(\Om\times\BR_+)^N \times L_p(\BR_+,W_q^1(\Om)),
\end{equation*}
with $\lim_{t\to0+}\|\Bu(t)-\Ba\|_{B_{q,p}^{2-2/p}(\Om)}=0$, satisfying the estimate:
\begin{equation}\label{est:main}
\|e^{\ga_0 t}(\pa_t\Bu,\Bu,\nabla\Bu,\nabla^2\Bu)\|_{L_p(\BR_+,L_q(\Om))}
+ \|e^{\ga_0 t} \te\|_{L_p(\BR_+,W_q^1(\Om))} \leq \de_0.
\end{equation}
\end{theo}

\begin{rema}
We discuss the equations \eqref{eq:2}-\eqref{eq:7} in Section \ref{sec:original} below.
\end{rema}

Next we introduce the maximal $L_p\text{-}L_q$ regularity with exponential stability
for the following linearized system associated with \eqref{eq:11}-\eqref{eq:15}: 
\begin{equation}\label{eq:linear}
\left\{\begin{aligned}
\pa_t\Bu-\Di\BT(\Bu,\Fp) &= \Bf && \text{in $\Om$, $t>0$,} \\
\di\Bu &= g && \text{in $\Om$, $t>0$,} \\
\BT(\Bu,\Fp)\Be_N &= \Bh && \text{on $\Ga$, $t>0$,} \\
\Bu &= 0 && \text{on $S$, $t>0$,} \\
\Bu|_{t=0}&=\Ba  && \text{in $\Om$.}
\end{aligned}\right.
\end{equation}
To this end, following \cite{MS17},
we introduce some function spaces related to the solvability of the divergence equation
$\di\Bu=g$ in $\Om$ with boundary condition $\Bu\cdot(-\Be_N)=0$ on $S$.
Let $1<q<\infty$ and $q'=q/(q-1)$.
Noting \cite[Lemma 2.3]{Abels06},
we can regard $W_{q',\Ga}^1(\Om)$ as a Banach space with norm $\|\nabla\cdot\|_{L_{q'}(\Om)}$,
which is denoted by $\wh W_{q',\Ga}^1(\Om)$.
Assume that $g\in L_q(\Om)$, and then one has
\begin{equation*}
|(g,\ph)_\Om| \leq C\|g\|_{L_q(\Om)}\|\nabla\ph\|_{L_{q'}(\Om)}
\quad \text{for any $\ph\in \wh W_{q',\Ga}^1(\Om)$,}
\end{equation*}
with some positive constant $C$ independent of $g$ and $\ph$.
This inequality implies that $g$ is an element of $\wh W_{q,\Ga}^{-1}(\Om)$,
where $\wh W_{q,\Ga}^{-1}(\Om)$ is the dual space of $\wh W_{q',\Ga}^1(\Om)$.
Here we see $g$ as a functional on $\{\nabla\ph \mid \ph\in \wh W_{q',\Ga}^1(\Om)\}\subset L_{q'}(\Om)^N$,
which, combined with Hahn-Banach's theorem, furnishes that
there is a $\BG\in L_q(\Om)^N$ such that
\begin{equation*}
\|g\|_{\wh W_{q,\Ga}^{-1}(\Om)} = \|\BG\|_{L_q(\Om)}, \quad
(g,\ph)_\Om=-(\BG,\nabla\ph)_\Om \quad \text{for all $\ph\in \wh W_{q',\Ga}^1(\Om)$.}
\end{equation*}
Let $[\BG]=\{\BG+\Bf \mid \Bf\in J_q(\Om)\}\in L_q(\Om)^N/J_q(\Om)$.
Then $L_q(\Om)\ni g \mapsto [\BG]\in L_q(\Om)^N/J_q(\Om)$
is well-defined, so that we denote $[\BG]$ by $\CG(g)$.
One especially notes that $\|\CG(g)\|_{L_q(\Om)^N/J_q(\Om)}=\|g\|_{\wh W_{q,\Ga}^{-1}(\Om)}$.
Thus, for any $g\in L_q(\Om)$ and any representative $\Fg$ of $\CG(g)$ regular enough, we have
\begin{equation*}
(\di\Fg,\ph)_\Om-(\Fg\cdot(-\Be_N),\ph)_S=(g,\ph)_\Om \quad \text{for all $\ph\in W_{q',\Ga}^1(\Om)$,}
\end{equation*}
which implies that $\Bu=\Fg$ solves the divergence equation mentioned above.

 Now we state the main result for \eqref{eq:linear} as follows:


\begin{theo}\label{theo:main2}
Let $1<p,q<\infty$ with $2/p+1/q\neq 1$ and $2/p+1/q\neq 2$. Then there exists a positive constant $\si_0$ such that,
for every $\Bf$, $g$, $\Bh$, and $\Ba$ satisfying
\begin{align*}
&e^{\si_0 t}\Bf \in L_p(\BR_+,L_q(\Om))^N, \quad
e^{\si_0 t}g \in {}_0 H_{q,p}^{1,1/2}(\BR_+\times\Om)\cap {}_0W_p^1(\BR_+,\wh W_{q,\Ga}^{-1}(\Om)), \\
&e^{\si_0 t}\Bh\in{}_0 H_{q,p}^{1,1/2}(\BR_+\times\Om)^N, \quad
\Ba\in D_{q,p}(\Om),
\end{align*}
the system \eqref{eq:linear} admits a unique solution $(\Bu,\Fp)\in W_{q,p}^{2,1}(\Om\times\BR_+)^N\times L_p(\BR_+,W_q^1(\Om))$
with $\lim_{t\to0+}\|\Bu(t)-\Ba\|_{B_{q,p}^{2-2/p}(\Om)}=0$.
In addition, 
the solution $(\Bu,\Fp)$ satisfies 
\begin{align*}
&\|e^{\si_0 t}(\pa_t\Bu,\Bu,\nabla\Bu,\nabla^2\Bu)\|_{L_p(\BR_+,L_q(\Om))}
+ \|e^{\si_0 t}\Fp\|_{L_p(\BR_+,W_q^1(\Om))} \\
&\leq
c_0\Big(\|e^{\si_0 t}\Bf\|_{L_p(\BR_+,L_q(\Om))}
+\|e^{\si_0 t}g\|_{W_p^1(\BR_+,\wh W_{q,\Ga}^{-1}(\Om))} \\
&\quad +\|e^{\si_0 t}(g,\Bh)\|_{{}_0H_{q,p}^{1,1/2}(\Om\times\BR_+)} 
+\|\Ba\|_{D_{q,p}(\Om)}\Big),
\end{align*}
with some positive constant $c_0\geq 1$ depending only on $N$, $d$, $p$, $q$, $\mu$, and $\si_0$.
\end{theo}

\section{Preliminaries}\label{sec:pre}
This section introduces some results concerning 
a time-shifted problem for \eqref{eq:linear}, a variational problem, and the Helmholtz decomposition on $\Om$.

\subsection{A time-shifted problem}
We consider in this subsection the following time-shifted linear system:
\begin{equation}\label{eq:shifted}
\left\{\begin{aligned}
\pa_t\Bu+2\de\Bu-\Di\BT(\Bu,\Fp) = \Bf & && \text{in $\Om$, $t\in\BR$,} \\
\di\Bu = g & && \text{in $\Om$, $t\in\BR$,} \\
\BT(\Bu,\Fp)\Be_N = \Bh & && \text{on $\Ga$, $t\in\BR$,} \\
\Bu = 0 & && \text{on $S$, $t\in\BR$,}
\end{aligned}\right.
\end{equation}
where $\de$ is a positive constant. 
More precisely, we prove

\begin{prop}\label{prop:shifted}
Let $1<p,q<\infty$ and $\de>0$. Then, for every
\begin{align*}
&\Bf \in L_{p,-\de}(\BR,L_q(\Om))^N, \quad 
g\in H_{q,p,-\de}^{1,1/2}(\Om\times\BR) \cap W_{p,-\de}^1(\BR,\wh W_{q,\Ga}^{-1}(\Om)),\\
&\Bh\in H_{q,p,-\de}^{1,1/2}(\Om\times\BR)^N,
\end{align*}
the system \eqref{eq:shifted} admits a unique solution $(\Bu,\Fp)$  with
\begin{equation*}
\Bu \in W_{q,p,-\de}^{2,1}(\Om\times\BR)^N, \quad
\Fp \in L_{p,-\de}(\BR,W_q^1(\Om)).
\end{equation*}
In addition, the following assertions hold true.
\begin{enumerate}[$(1)$]
\item\label{prop:shifted_1}
The solution $(\Bu,\Fp)$ satisfies the estimate:
\begin{align*}
&\|e^{\de t}(\pa_t\Bu,\Bu,\nabla\Bu,\nabla^2\Bu)\|_{L_p(\BR,L_q(\Om))} + \|e^{\de t}\Fp\|_{L_p(\BR,W_q^1(\Om))} \\
&\leq
C\Big(\|e^{\de t}\Bf\|_{L_p(\BR,L_q(\Om))}
+\|e^{\de t}(\pa_t g,g)\|_{L_p(\BR,\wh W_{q,\Ga}^{-1}(\Om))}
+ \|e^{\de t}(g,\Bh)\|_{H_{q,p}^{1,1/2}(\Om\times\BR)}\Big),
\end{align*}
with some positive constant $C=C_{N,d,p,q,\de,\mu}$.
\item\label{prop:shifted_2}
If $\Bf$, $g$, and $\Bh$ vanish for $t<0$, then $\Bu$ also vanishes for $t<0$.
\end{enumerate}
\end{prop}

\begin{rema}\label{rema:trace}
Proposition \ref{prop:shifted} \eqref{prop:shifted_2} implies $\Bu|_{t=0}=0$ in $L_q(\Om)^N$,
provided that $\Bf$, $g$, and $\Bh$ vanish for $t<0$.
\end{rema}

To prove Proposition \ref{prop:shifted}, we start with 
\begin{equation}\label{eq:0data}
\left\{\begin{aligned}
\pa_t\BU-\Di\BT(\BU,P) &= \BF && \text{in $\Om$, $t\in\BR$,} \\
\di\BU &= G && \text{in $\Om$, $t\in\BR$,} \\
\BT(\BU,P)\Be_N &= \BH && \text{on $\Ga$, $t\in\BR$,} \\
\BU &= 0 && \text{on $S$, $t\in\BR$.}
\end{aligned}\right.
\end{equation}
Concerning this system, we have

\begin{lemm}\label{lemm:0data}
Let $1<p,q<\infty$ and $\ga>0$. Then, for every
\begin{align*}
&\BF \in L_{p,\ga}(\BR,L_q(\Om))^N, \quad
G \in H_{q,p,\ga}^{1,1/2}(\Om\times\BR) \cap W_{p,\ga}^1(\BR,\wh W_{q,\Ga}^{-1}(\Om)), \\
&\BH \in H_{q,p,\ga}^{1,1/2}(\Om\times\BR)^N,
\end{align*}
the system \eqref{eq:0data} admits a unique solution $(\BU,P)$ with
\begin{equation*}
\BU \in W_{q,p,\ga}^{2,1}(\Om\times\BR)^N, \quad
P \in L_{p,\ga}(\BR,W_q^1(\Om)).
\end{equation*}
In addition, the following assertions hold true.
\begin{enumerate}[$(1)$]
\item\label{lemm:0data_1}
The solution $(\BU,P)$ satisfies the estimate:
\begin{align*}
&\|e^{-\ga t}(\pa_t\BU,\BU,\nabla\BU,\nabla\BU^2)\|_{L_p(\BR,L_q(\Om))}
+\|e^{-\ga t}P\|_{L_p(\BR,W_q^1(\Om))} \\
&\leq
C\Big(\|e^{-\ga t}\BF\|_{L_p(\BR,L_q(\Om))} +\|e^{-\ga t}(\pa_t G,G)\|_{L_p(\BR,\wh W_{q,\Ga}^{-1}(\Om))}\\
&\quad
+\|e^{-\ga t}(G,\BH)\|_{H_{q,p}^{1,1/2}(\Om\times\BR)}\Big),
\end{align*}
with some positive constant $C=C_{N,d,p,q,\ga,\mu}$.
\item\label{lemm:0data_2}
If $\BF$, $G$, and $\BH$ vanish for $t<0$,
then $\BU$ also vanishes for $t<0$.
\end{enumerate}
\end{lemm}

\begin{proof}
This lemma was proved in \cite[Theorem 2.1]{Saito15b}.
\end{proof}

\begin{proof}[Proof of Proposition $\ref{prop:shifted}$]
In order to apply Lemma \ref{lemm:0data} with $\ga=\de$ to \eqref{eq:shifted}, 
we set $\BF=e^{2\de t}\Bf$, $G=e^{2\de t}g$, and $\BH=e^{2\de t}\Bh$.
It then is clear that
\begin{align*}
&\BF \in L_{p,\de}(\BR,L_q(\Om))^N, \quad
G \in H_{q,p,\de}^{1,1/2}(\Om\times\BR)\cap W_{p,\de}^1(\BR,\wh W_{q,\Ga}^{-1}(\Om)), \\
&\BH \in H_{q,p,\de}^{1,1/2}(\Om\times\BR)^N.
\end{align*}
Thus, by Lemma \ref{lemm:0data}, 
we have a solution
$(\BU,P)\in W_{q,p,\de}^{2,1}(\Om\times\BR)^N\times L_{p,\de}(\BR,W_q^1(\Om))$ to \eqref{eq:0data},
which satisfies
\begin{align*}
&\|e^{-\de t}(\pa_t\BU,\BU,\nabla\BU,\nabla\BU^2)\|_{L_p(\BR,L_q(\Om))}
+\|e^{-\de t}P\|_{L_p(\BR,W_q^1(\Om))} \\
&\leq
C\left(\|e^{\de t}\Bf\|_{L_p(\BR,L_q(\Om))}
+\|e^{\de t}(\pa_t g,g)\|_{L_p(\BR, \wh W_{q,\Ga}^{-1}(\Om))}
+\|e^{\de t}(g,\Bh)\|_{H_{q,p}^{1,1/2}(\Om\times\BR)}\right)
\end{align*}
for some positive constant $C=C_{N,d,p,q,\de,\mu}$.
Let $(\Bu,\Fp)=(e^{-2\de t}\BU,e^{-2\de t}P)$,
and then $(\Bu,\Fp)$ solves the system \eqref{eq:shifted}
and satisfies the required estimate of Proposition \ref{prop:shifted} \eqref{prop:shifted_1} by the last inequality.
The other assertions immediately follow from Lemma \ref{lemm:0data},
which completes the proof of Proposition \ref{prop:shifted}.
\end{proof}

\subsection{A variational problem}
Let $1<q<\infty$ and $q'=q/(q-1)$.
This subsection is concerned with the following variational problem:
\begin{equation}\label{eq:vari}
(\nabla u,\nabla\ph)_\Om=(\Bf,\nabla\ph)_\Om \quad \text{for all $\ph\in W_{q',\Ga}^1(\Om)$,}
\end{equation}
which is the so-called {\it weak Dirichlet-Neumann problem}.
Our aim in this subsection is to prove 

\begin{prop}\label{prop:weakDN}
Let $1<q<\infty$ and $q'=q/(q-1)$.
Then, for every $\Bf\in L_q(\Om)^N$,
the variational problem \eqref{eq:vari} admits a unique solution $u\in W_{q,\Ga}^1(\Om)$,
and also $\|u\|_{W_q^1(\Om)}\leq C_{N,d,q}\|\Bf\|_{L_q(\Om)}$
for a positive constant $C_{N,d,q}$.
In this case, we define the solution operator $Q_q$ from $L_q(\Om)^N$ to $W_{q,\Ga}^1(\Om)$
by $Q_q\Bf = u$. 
\end{prop}

\begin{proof}
Since $C_0^\infty(\Om)$ is dense in $L_q(\Om)$,
it suffices to consider the case where $\Bf=(f_1,\dots,f_N)^\SST\in C_0^\infty(\Om)^N$.

We first consider a strong problem associated with \eqref{eq:vari} as follows:
\begin{equation}\label{eq:strong}
\left\{\begin{aligned}
\De u &= \di\Bf && \text{in $\Om$,} \\
u &=0 && \text{on $\Ga$,} \\
\pa_N u&=0 && \text{on $S$.}
\end{aligned}\right.
\end{equation}
For $g=g(x',x_N)$ defined on $\BR_+^N$, 
let $g^o$ and $g^e$ be the odd extension of $g$ and the even extension of $g$, respectively, i.e.
\begin{equation*}
g^o=
\left\{\begin{aligned}
&g(x',x_N) && (x_N>0), \\
&-g(x',-x_N) && (x_N<0)
\end{aligned}\right., \quad
g^e=
\left\{\begin{aligned}
&g(x',x_N) && (x_N>0), \\
&g(x',-x_N) && (x_N<0).
\end{aligned}\right.
\end{equation*}
Regarding $\Bf$ as an element of $C_0^\infty(\BR_+^N)^N$,
we set
\begin{equation}\label{170223_1}
\BF=(F_1,\dots,F_N)^\SST=(f_1^e,\dots,f_{N-1}^e,f_N^o)^\SST.
\end{equation}
Let $\CF[g](\xi)$ be the Fourier transform of $g=g(x)$ and
$\CF_{\xi}^{-1}[h](x)$ the inverse Fourier transform of $h=h(\xi)$, i.e.
\begin{equation*}
\CF[g](\xi)=\int_{\BR^N}e^{-ix\cdot\xi}g(x)\intd x,\quad
\CF_{\xi}^{-1}[h](x)=\frac{1}{(2\pi)^N}\int_{\BR^N}e^{ix\cdot\xi}h(\xi)\intd\xi.
\end{equation*}

Now we define a function $v=v(x)$ on $\BR^N$ by
\begin{equation}\label{170223_2}
v=-\CF_\xi^{-1}\left[\frac{1}{|\xi|^2}\CF[\di\BF](\xi)\right](x)
=-\sum_{j=1}^{N}\CF_\xi^{-1}\left[\frac{i\xi_j}{|\xi|^2}\CF[F_j](\xi)\right](x).
\end{equation}
It then holds that $\De v=\di\BF$ in $\BR^N$, which implies that $\De v=\di\Bf$ in $\Om$.
On the other hand, one has for $k,l=1,\dots,N$
\begin{equation*}
\pa_k v = \sum_{j=1}^{N}\CF_\xi^{-1}\left[\frac{\xi_j\xi_k}{|\xi|^2}\CF[F_j](\xi)\right](x), \quad
\pa_k\pa_l v = \CF_\xi^{-1}\left[\frac{\xi_k\xi_l}{|\xi|^2}\CF[\di\BF](\xi)\right](x),
\end{equation*}
which, combined with the Fourier multiplier theorem of Mikhlin (cf. \cite[Appendix Theorem 2]{Mikhlin65}), furnishes that
\begin{equation}\label{170223_5}
\|\nabla v\|_{L_q(\BR^N)} \leq C_{N,q}\|\Bf\|_{L_q(\Om)}, \quad
\|\nabla^2 v\|_{L_q(\BR^N)}\leq C_{N,q}\|\nabla\Bf \|_{L_q(\Om)},
\end{equation}
for some positive constant $C_{N,q}$.

Next we estimate $\|v\|_{L_q(\Om)}$.
Let $\wh g(\xi',x_N)$ be the partial Fourier transform of $g=g(x',x_N)$
with respect to $x'$ and
$\CF_{\xi'}^{-1}[h(\xi',x_N)](x')$ the inverse partial Fourier transform of $h=h(\xi',x_N)$
with respect to $\xi'$, i.e.
\begin{align*}
&\wh g(\xi',x_N) = \int_{\BR^{N-1}}e^{-ix'\cdot\xi'}g(x',x_N)\intd x', \\
&\CF_{\xi'}^{-1}[h(\xi',x_N)](x') = \frac{1}{(2\pi)^{N-1}}\int_{\BR^{N-1}}e^{ix'\cdot\xi'}h(\xi',x_N)\intd\xi'.
\end{align*}
We then observe that, for $j=1,\dots,N-1$,
\begin{align*}
\CF[F_j](\xi) = \int_0^d \left(e^{-i y_N\xi_N}+e^{i y_N\xi_N}\right)\wh f_j(\xi',y_N)\intd y_N,
\end{align*}
and that
\begin{equation*}
\CF[F_N](\xi)=\int_0^d\left(e^{-i y_N\xi_N}-e^{i y_N\xi_N}\right)\wh f_N(\xi',y_N)\intd y_N.
\end{equation*}
Inserting these formulas into \eqref{170223_2} yields
\begin{align*}
&v= v(x',x_N) =\\
&-\sum_{j=1}^{N-1}\int_0^d\CF_{\xi'}^{-1}
\left[\left(\frac{1}{2\pi}\int_{-\infty}^\infty\frac{e^{i(x_N-y_N)\xi_N}+e^{(x_N+y_N)\xi_N}}{|\xi|^2}\right)\intd\xi_N
i\xi_j\wh f_j(\xi',y_N)\right](x')\intd y_N \\
&-\int_0^d\CF_{\xi'}^{-1}
\left[\left(\frac{1}{2\pi}\int_{-\infty}^\infty\frac{i\xi_N(e^{i(x_N-y_N)\xi_N}+e^{(x_N+y_N)\xi_N})}{|\xi|^2}\right)\intd\xi_N
\wh f_N(\xi',y_N)\right](x')\intd y_N.
\end{align*}
On the other hand, it holds by the residue theorem that for $a\in\BR\setminus\{0\}$
\begin{equation}\label{170223_8}
\int_{-\infty}^\infty\frac{e^{ia\xi_N}}{|\xi|^2}\intd\xi_N
=\frac{\pi e^{-|a||\xi'|}}{|\xi'|}, \quad
\int_{-\infty}^\infty\frac{i\xi_N e^{ia\xi_N}}{|\xi|^2}\intd\xi_N
=-\pi e^{-|a||\xi'|}{\rm sign}(a),
\end{equation}
where ${\rm sign}(a)=1$ when $a>0$ and ${\rm sign}(a)=-1$ when $a<0$.
We combine \eqref{170223_8} with the above formula of $v$ in order to obtain
\begin{align*}
&v=-\frac{1}{2}\sum_{j=1}^{N-1}\int_0^d
\CF_{\xi'}^{-1}\left[\frac{i\xi_j}{|\xi'|}
\left(e^{-|x_N-y_N||\xi'|}+e^{-(x_N+y_N)|\xi'|}\right)\wh f_j(\xi',y_N)\right](x')\intd y_N \\
&+\frac{1}{2}\int_0^d\CF_{\xi'}^{-1}\left[
\left({\rm sign}(x_N-y_N)e^{-|x_N-y_N||\xi'|}
+e^{-(x_N+y_N)|\xi'|}\right)\wh f_N(\xi',y_N)\right](x')\intd y_N,
\end{align*}
which implies that $\|v\|_{L_q(\Om)}\leq C_{N,d,q}\|\Bf\|_{L_q(\Om)}$
in the same manner as in the proof of \cite[pages 1897-1898]{Saito15b}.
Hence, together with \eqref{170223_5}, one has
\begin{equation}\label{170223_6}
\|v\|_{W_q^1(\Om)} \leq C_{N,d,q}\|\Bf\|_{L_q(\Om)}, \quad \|v\|_{W_q^2(\Om)}\leq C_{N,d,q}\|\Bf\|_{W_q^1(\Om)}.
\end{equation}



We here prove that $\pa_N v=0$ on $\BR_0^N$.
Noting $\di\BF=(\di\Bf)^e$ by the definition \eqref{170223_1} and setting $z=\di\Bf$,
we have
\begin{align*}
\CF[\di\BF](\xi)=\CF[z^e](\xi)
=\int_0^d\left(e^{-iy_N\xi_N}+e^{iy_N\xi_N}\right)\wh z(\xi',y_N)\intd y_N.
\end{align*}
By this formula and \eqref{170223_2}, one obtains
\begin{align*}
&\pa_N v = (\pa_N v)(x',x_N)=\\
&-\int_0^d\CF_{\xi'}^{-1}\left[
\left(\frac{1}{2\pi}\int_{-\infty}^\infty
\frac{i\xi_N(e^{i(x_N-y_N)\xi_N}+e^{i(x_N+y_N)\xi_N})}{|\xi|^2}\intd\xi_N\right)\wh z(\xi',y_N)\right](x')\intd y_N,
\end{align*}
which, combined with \eqref{170223_8}, furnishes $(\pa_N v)(x',0)=0$.
Let $u=v+w$ in \eqref{eq:strong}, and thus \eqref{eq:strong} is reduced to
\begin{equation}\label{170223_10}
\left\{\begin{aligned}
\De w &= 0 && \text{in $\Om$,} \\
w&=-v && \text{on $\Ga$,} \\
\pa_N w&=0 && \text{on $S$.}
\end{aligned}\right.
\end{equation}

From now on, we solve the system \eqref{170223_10}. 
Applying the partial Fourier transform to \eqref{170223_10} yields
\begin{equation*}
\left\{\begin{aligned}
(\pa_N^2-|\xi'|^2)\wh w(\xi',x_N)&= 0 && (0<x_N<d), \\
\wh w(\xi',d)&=-\wh v(\xi',d), \\
\pa_N\wh w(\xi',0)&=0.
\end{aligned}\right.
\end{equation*}
One then solves this system as ordinary differential equations with respect to $x_N$
in order to obtain
\begin{equation*}
w=w(x',x_N)=-\sum_{k=1}^2
\CF_{\xi'}^{-1}\left[\frac{e^{-|\xi'|(d+(-1)^k x_N)}}{1+e^{-2|\xi'|d}}
\wh v(\xi',d)\right](x').
\end{equation*}
Let $\ph=\ph(s)$ be a function in $C^\infty(\BR)$ such that $0\leq \ph\leq 1$ and
\begin{equation*}
\ph(s)=\left\{\begin{aligned}
&1 && \text{for $s\geq 2d/3$}, \\
&0 && \text{for $s\leq d/3$},
\end{aligned}\right.
\end{equation*} 
and note that for functions $f(x_N)$, $g(x_N)$ and for $k=1,2$
\begin{equation*}
f(x_N)g(d)
=\int_0^d\frac{d}{d y_N}\left\{\ph(y_N)f(x_N+(-1)^k(y_N-d))g(y_N)\right\}\intd y_N. 
\end{equation*}
Combining these identities with the above formula of $w$, we obtain
\begin{align*}
w
&=-\sum_{k=1}^2\int_0^d
\CF_{\xi'}^{-1}\left[\frac{\dot\ph(y_N) e^{-|\xi'|(y_N+(-1)^k x_N)}}{1+e^{-2|\xi'|d}}
\wh v(\xi',y_N)\right](x')\intd y_N \\
&+\sum_{k=1}^2\int_0^d
\CF_{\xi'}^{-1}\left[\frac{\ph(y_N)|\xi'|  e^{-|\xi'|(y_N+(-1)^k x_N)}}{1+e^{-2|\xi'|d}}
\wh v(\xi',y_N)\right](x')\intd y_N \\
&-\sum_{k=1}^2\int_0^d
\CF_{\xi'}^{-1}\left[\frac{\ph(y_N) e^{-|\xi'|(y_N+(-1)^k x_N)}}{1+e^{-2|\xi'|d}}
\wh{\pa_N v}(\xi',y_N)\right](x')\intd y_N  \\
&=:I_1+I_2+I_3,
\end{align*}
where $\dot\ph(y_N)=(d\ph/dy_N)(y_N)$.
Noting $|\xi'|=|\xi'|^2/|\xi'|=-\sum_{j=1}^{N-1}(i\xi_j)^2/|\xi'|$,
we can write $I_2$ as 
\begin{equation*}
I_2 = -\sum_{j=1}^{N-1}\sum_{k=1}^2\int_0^d
\CF_{\xi'}^{-1}\left[\frac{\ph(y_N)i\xi_j e^{-|\xi'|(y_N+(-1)^k x_N)}}{|\xi'|(1+e^{-2|\xi'|d})}
\wh{\pa_j v}(\xi',y_N)\right](x')\intd y_N.
\end{equation*}
The following lemma was essentially proved in Lemma \cite[Lemma 5.5]{Saito15b}.

\begin{lemm}\label{est:w}
Let $1<q<\infty$. Assume that $m(\xi')$ satisfies
\begin{equation*}
|\pa_{\xi'}^{\al'} m(\xi')|\leq C_{\al'} |\xi'|^{-|\al'|} \quad (\xi'\in\BR^{N-1}\setminus\{0\})
\end{equation*}
for any multi-index $\al'\in\BN_0^{N-1}$, and set for $k=1,2$
\begin{align*}
[L_k f](x)=\int_0^d\CF_{\xi'}^{-1}\left[\dot\ph (y_N)m(\xi')e^{-|\xi'|(y_N+(-1)^k x_N)}\wh f(\xi',y_N)\right](x')\intd y_N, \\
[M_k f](x)=\int_0^d\CF_{\xi'}^{-1}\left[\ph(y_N)m(\xi')e^{-|\xi'|(y_N+(-1)^k x_N)}\wh f(\xi',y_N)\right](x')\intd y_N.
\end{align*}
Then,  for any $f\in L_q(\Om)$, we have
\begin{equation*}
\|(L_k f, M_k f)\|_{W_q^1(\Om)} \leq C_{N,d,q}\|f\|_{L_q(\Om)} \quad (k=1,2).
\end{equation*}
\end{lemm}

It is known by e.g. \cite[Section 5]{SS12} that
for any multi-index $\al'\in\BN_0^{N-1}$
\begin{equation}\label{170525_1}
\big|\pa_{\xi'}^{\al'}(i\xi_j|\xi'|^{-1})\big| 
+\big|\pa_{\xi'}^{\al'}e^{-a|\xi'|}\big|\leq C_{\al'}|\xi'|^{-|\al'|} \quad (\xi'\in\BR^{N-1}\setminus\{0\}),
\end{equation}
where $j=1,\dots,N-1$ and $a\geq 0$, with some positive constant $C_{\al'}$ independent of $\xi'$ and $a$.
In order to estimate $(1+e^{-2|\xi'|d})^{-1}$, we introduce Bell's formula for derivatives of the composite function of
$f(t)$ and $t=g(\xi')$ as follows: For any multi-index $\al'\in\BN_0^{N-1}$,
\begin{equation*}
\pa_{\xi'}^{\al'} f(g(\xi'))=\sum_{l=1}^{|\al'|} f^{(l)}(g(\xi'))
\sum_{\stackrel{\text{\scriptsize{$\al_1'+\dots+\al_l'=\al'$,}}}{|\al_m'|\geq 1}}
\Ga_{\al_1',\dots,\al_l'}^{\al'}(\pa_{\xi'}^{\al_1'}g(\xi'))\dots(\pa_{\xi'}^{\al_l'}g(\xi'))
\end{equation*}
with suitable coefficients $\Ga_{\al_1',\dots,\al_l'}^{\al'}$,
where $f^{(l)}(t)$ is the $l$th derivative of $f(t)$. 
By Bell's formula with $f(t)=t^{-1}$ and $t=g(\xi')=1+e^{-2|\xi'|d}$ and by \eqref{170525_1}, 
\begin{equation}\label{170525_3}
\big|\pa_{\xi'}^{\al'}(1+e^{-2|\xi'|d})^{-1}\big|
\leq C_{\al'}\sum_{l=1}^{|\al'|}\big|1+e^{-2|\xi'|d}|^{-(l+1)}\big|\xi'|^{-|\al'|}\leq C_{\al'}|\xi'|^{-|\al'|}
\end{equation}
for any $\xi'\in\BR^{N-1}\setminus\{0\}$ and any multi-index $\al'\in\BN_0^{N-1}$.
One thus obtains
$\|(I_1,I_3)\|_{W_q^1(\Om)} \leq C_{N,d,q}\|v\|_{W_q^1(\Om)}$
by Lemma \ref{est:w} and \eqref{170525_3}.
In addition, it holds by Leibniz's rule, \eqref{170525_1}, and \eqref{170525_3} that
\begin{equation*}
\left|\pa_{\xi'}^{\al'}\left(\frac{i\xi_j}{|\xi'|(1+e^{-2|\xi'|d})}\right)\right|\leq C_{\al'}|\xi'|^{-|\al'|}
\quad (\xi'\in\BR^{N-1}\setminus\{0\}),
\end{equation*}
which, combined with Lemma \ref{est:w}, furnishes 
$\|I_2\|_{W_q^1(\Om)} \leq C_{N,d,q}\|v\|_{W_q^1(\Om)}$.
%
%
%
Analogously, we can prove that
$\|\nabla(I_1,I_2,I_3)\|_{W_q^1(\Om)} \leq C_{N,d,q}\|\nabla v\|_{W_q^1(\Om)}$.
Summing up these inequalities for $I_1$, $I_2$, and $I_3$, one has by \eqref{170223_6}
\begin{equation}\label{170224_1}
\|w\|_{W_q^1(\Om)} \leq C_{N,d,q}\|\Bf\|_{L_q(\Om)}, \quad
\|w\|_{W_q^2(\Om)} \leq C_{N,d,q}\|\Bf\|_{W_q^1(\Om)}.
\end{equation}

It now holds that $u=v+w$ solves \eqref{eq:strong} and satisfies
$\|u\|_{W_q^1(\Om)}\leq C_{N,d,q}\|\Bf\|_{L_q(\Om)}$ by \eqref{170223_6} and \eqref{170224_1}.
Clearly, such an $u$ is a solution to \eqref{eq:vari}. 

Finally, we prove the uniqueness of solutions to \eqref{eq:vari}.
Let $u\in W_{q,\Ga}^1(\Om)$ be a solution to \eqref{eq:vari} with $\Bf=0$,
and let $\Phi\in C_0^\infty(\Om)^N$.
Since $\Phi \in L_{q'}(\Om)^N$, there is a $v\in W_{q',\Ga}^1(\Om)$ such that
\begin{equation*}
(\nabla v,\nabla\psi)_\Om=(\Phi,\nabla\psi)_\Om\quad \text{for all $\psi\in W_{q,\Ga}^1(\Om)$.}
\end{equation*}
In this equation, we set $\psi=u$ in order to obtain
\begin{equation*}
(\Phi,\nabla u)_\Om = (\nabla v,\nabla u)_\Om = 0,
\end{equation*}
which implies that $u$ is a constant.
Hence, we have $u=0$, because $u=0$ on $\Ga$.
This completes the proof of the Proposition \ref{prop:weakDN}.
\end{proof}

\subsection{Helmholtz decomposition on $\Om$}\label{subsec:HD}
We here introduce the Helmholtz decomposition on $\Om$.
Let $G_q(\Om) = \{\nabla\te \mid \te \in W_{q,\Ga}^1(\Om)\}$,
and then one has
\begin{prop}\label{prop:decomp}
Let $1<q<\infty$.
Then the following assertions hold true.
\begin{enumerate}[$(1)$]
\item\label{prop:decomp1}
$L_q(\Om)^N = J_q(\Om) \oplus G_q(\Om)$.	
\item\label{prop:decomp2}
Let $P_q$ be the projection from $L_q(\Om)^N$ to $J_q(\Om)$,
and let $Q_q$ be the solution operator of Proposition $\ref{prop:weakDN}$
from $L_q(\Om)^N$ to $W_{q,\Ga}^1(\Om)$.
Then, for any $\Bf \in L_q(\Om)^N$,
we have $\Bf = P_q\Bf + \nabla Q_q\Bf\in J_q(\Om)\oplus G_q(\Om)$ and
\begin{equation}\label{170521_1}
\|P_q\Bf\|_{L_q(\Om)}+\|Q_q\Bf\|_{W_q^1(\Om)} \leq C_{N,d,q}\|\Bf\|_{L_q(\Om)}
\end{equation}
for some positive constant $C_{N,d,q}$.
\end{enumerate}
\end{prop}

\begin{proof}
(1). Let $\Bf\in L_q(\Om)^N$. By Proposition \ref{prop:weakDN},
we have a unique solution $u= Q_q\Bf\in W_{q,\Ga}^1(\Om)$ to the variational problem:
\begin{equation*}
(\nabla u,\nabla\ph)_\Om=(\Bf,\nabla\ph)_\Om \quad \text{for all $\ph\in W_{q',\Ga}^1(\Om)$,}
\end{equation*}
where $q'=q/(q-1)$.
Setting $\Bg = \Bf -\nabla Q_q\Bf$, 
we observe that
$\Bg \in J_q(\Om)$ 
and $\Bf = \Bg + \nabla Q_q\Bf\in J_q(\Om)+G_q(\Om)$.

Next we prove that $L_q(\Om)^N=J_q(\Om)\oplus G_q(\Om)$.
To this end, let $\Bf \in J_q(\Om)\cap G_q(\Om)$. 
We then have
\begin{equation*}
(\nabla Q_q\Bf,\nabla\ph)_\Om=(\Bf,\nabla\ph)_\Om = 0 \quad \text{for all $\ph\in W_{q',\Ga}^1(\Om)$,}
\end{equation*}
which, combined with the uniqueness of Proposition \ref{prop:weakDN},
furnishes $Q_q\Bf=0$.
On the other hand, since there is a $\te\in W_{q,\Ga}^1(\Om)$ such that $\Bf=\nabla\te$,
we have
\begin{equation*}
(\nabla Q_q\Bf,\nabla\ph)_\Om=(\Bf,\nabla\ph)_\Om=(\nabla\te,\nabla\ph)_\Om \quad
\text{for all $\ph\in W_{q',\Ga}^1(\Om)$.}
\end{equation*}
By the uniqueness of Proposition \ref{prop:weakDN} again,
it holds that $\te= Q_q\Bf =0$.
Thus $\Bf=0$, which implies $L_q(\Om)^N=J_q(\Om)\oplus G_q(\Om)$.

(2). For any $\Bf\in L_q(\Om)^N$, the projection $P_q$ is given by $P_q\Bf=\Bf-\nabla Q_q\Bf\in J_q(\Om)$ 
as was seen in the proof of \eqref{prop:decomp1}. 
Thus the first assertion clearly holds true,
and also the second assertion \eqref{170521_1} follows from Proposition \ref{prop:weakDN} immediately.
This completes the proof of Proposition \ref{prop:decomp}. 
\end{proof}

\section{Generation of the Stokes semigroup}\label{sec:sg}
Our aim in this section is to construct an analytic $C_0$-semigroup associated with
\begin{equation}\label{eq:sg}
\left\{\begin{aligned}
\pa_t\Bu-\Di\BT(\Bu,\Fp) &= 0 && \text{in $\Om$, $t>0$,} \\
\di\Bu&=0 && \text{in $\Om$, $t>0$,} \\ 
\BT(\Bu,\Fp)\Be_N &= 0 && \text{on $\Ga$, $t>0$,} \\
\Bu &= 0 && \text{on $S$, $t>0$,} \\
\Bu|_{t=0} &= \Ba && \text{in $\Om$.}
\end{aligned}\right.
\end{equation}
To this end, we start with the following resolvent problem: 
\begin{equation}\label{eq:resolvent}
\left\{\begin{aligned}
\la\Bv-\Di\BT(\Bv,\Fq) &=\Bf && \text{in $\Om$,} \\
\di\Bv&=0 && \text{in $\Om$,} \\
\BT(\Bv,\Fq)\Be_N &= 0 && \text{on $\Ga$,} \\
\Bv &= 0 && \text{on $S$,}
\end{aligned}\right.
\end{equation}
with the resolvent parameter $\la\in\Si_\ep=\{\om\in\BC\setminus\{0\} \mid|\arg\om|<\pi-\ep\}$ for $0<\ep<\pi/2$.
The following lemma was proved in \cite[Theorem 1.1]{Abels06}.

\begin{lemm}\label{lemm:Abels06}
Let $1<q<\infty$ and $0<\ep<\pi/2$.
Then there exists a positive number $\om_1$ such that,
for every $\la \in \Si_\ep\cup\{\om\in\BC \mid |\om|<\om_1\}$
and $\Bf\in L_q(\Om)^N$,
there is a unique solution $(\Bv,\Fq)\in W_q^2(\Om)^N \times W_q^1(\Om)$ to \eqref{eq:resolvent}.
In addition, 
\begin{equation*}
|\la|\|\Bv\|_{L_q(\Om)}+|\la|^{1/2}\|\nabla\Bv\|_{L_q(\Om)}+\|\Bv\|_{W_q^2(\Om)}
+\|\Fq\|_{W_q^1(\Om)} 	\leq C\|\Bf\|_{L_q(\Om)}
\end{equation*}
for any $\la \in \Si_\ep\cup\{\om\in\BC \mid |\om|<\om_1\}$
with a positive constant $C=C_{N,d,q,\ep,\mu,\om_1}$.
\end{lemm}

Let $1<q<\infty$ and $q'=q/(q-1)$.
By Proposition \ref{prop:weakDN}, we see that, 
for any $\Bf \in L_q(\Om)^N$ and $g\in W_q^{1-1/q}(\Ga)$,
there exists a unique solution $\wt{\Fq}\in W_{q,\Ga}^1(\Om)$ to the variational problem:
\begin{equation*}
(\nabla\wt{\Fq},\ph)_\Om
= (\Bf-\nabla\wt{g},\nabla\ph)_\Om \quad
\text{for all $\ph\in W_{q',\Ga}^1(\Om)$},
\end{equation*}
where $\wt{g}$ is an extension of $g$ satisfying
$\|\wt{g}\|_{W_q^1(\Om)}\leq C\|g\|_{W_q^{1-1/q}(\Ga)}$ for some positive constant $C$ independent of $g$, $\wt g$.
Furthermore, the solution $\wt{\Fq}$ satisfies 
\begin{align*}
\|\wt{\Fq}\|_{W_q^1(\Om)}
&\leq C_{N,d,q}(\|\Bf\|_{L_q(\Om)}+\|\nabla\wt{g}\|_{L_q(\Om)}) \\
&\leq C_{N,d,q}(\|\Bf\|_{L_q(\Om)}+\|g\|_{W_q^{1-1/q}(\Ga)}).
\end{align*}
Setting $\Fq=\wt{\Fq}+\wt{g}\in W_{q,\Ga}^1(\Om)+W_q^1(\Om)$ yields that
\begin{equation*}
(\nabla\Fq,\nabla\ph)_\Om = (\Bf,\nabla\ph)_\Om \quad
\text{for all $\ph\in W_{q',\Ga}^1(\Om)$}, \quad
\Fq =g \quad \text{on $\Ga$},
\end{equation*}
and also $\|\Fq\|_{W_q^1(\Om)}\leq C_{N,d,q}(\|\Bf\|_{L_q(\Om)}+\|g\|_{W_q^{1-1/q}(\Ga)})$.
From this viewpoint,
we define an operator $K: W_q^2(\Om)^N \ni \Bv \mapsto K(\Bv)\in W_{q,\Ga}^1(\Om)+W_q^1(\Om)$ as follows:
\begin{align*}
&(\nabla K(\Bv),\nabla\ph)_\Om =
(\Di(\mu\BD(\Bv))-\nabla\di\Bv,\nabla\ph)_\Om \quad
\text{for all $\ph\in W_{q',\Ga}^1(\Om)$}, \\
&K(\Bv) = \Be_N\cdot(\mu\BD(\Bv)\Be_N)-\di\Bv \quad \text{on $\Ga$}.
\end{align*}
Then $K(\Bv)$ satisfies $\|K(\Bv)\|_{W_q^1(\Om)}\leq C_{N,d,q,\mu}\|\Bv\|_{W_q^2(\Om)}$
with some positive constant $C_{N,q,d,\mu}$ independent of $\Bv$ and $\ph$.

At this point, we introduce a result concerning
the weak Dirichlet-Neumann problem with resolvent parameter $\la$.

\begin{prop}\label{prop:weakDN_re}
Let $1<q<\infty$ and $0<\ep<\pi/2$.
Then there exists a positive number $\om_2$ such that,
for every $\la\in\Si_\ep\cup\{\om\in\BC\mid|\om|<\om_2\}$
and $\Bf\in L_q(\Om)^N$,
there is a unique solution $u\in W_{q,\Ga}^1(\Om)$
to the variational problem:
\begin{equation}\label{170224_3}
(\la u,\ph)_\Om+(\nabla u,\nabla\ph)_\Om=(\Bf,\nabla\ph)_\Om
\quad \text{for all $\ph\in W_{q',\Ga}^1(\Om)$,}
\end{equation}
where $q'=q/(q-1)$. 
\end{prop}

\begin{proof}
The case $\la=0$ was already proved in Proposition \ref{prop:weakDN}.
Then, by a small perturbation method,
we can prove that there exists a positive constant $\om_2$
such that, for every $\la\in\{\om\in\BC\mid|\om|<\om_2\}$ and
$\Bf\in L_q(\Om)^N$, 
\eqref{170224_3} admits a unique solution $u\in W_{q,\Ga}^1(\Om)$.
In the case $\la\in\Si_\ep$ with $|\la|\geq\om_2/2$,
we consider the strong problem with resolvent parameter $\la$ for \eqref{eq:strong}.
One can construct solutions to the strong problem 
in a similar way to Proposition \ref{prop:weakDN} (cf. also \cite{Saito15b}). 
The uniqueness follows from the existence of solutions for a dual problem
as was seen in the proof of Proposition \ref{prop:weakDN}.
This completes the proof of the proposition.
\end{proof}

We now consider the reduced resolvent problem: 
\begin{equation}\label{eq:reduced}
\left\{\begin{aligned}
\la\Bv-\Di\BT(\Bv,K(\Bv)) &= \Bf && \text{in $\Om$,} \\
\BT(\Bv,K(\Bv))\Be_N &= 0 && \text{on $\Ga$,} \\
\Bv &= 0 && \text{on $S$}.
\end{aligned}\right.
\end{equation}
Then the following proposition holds.

\begin{prop}\label{prop:equiv}
Let $1<q<\infty$ and $0<\ep<\pi/2$.
Assume that $\om_1$ and $\om_2$ are, respectively, the same positive constants as in 
Lemma $\ref{lemm:Abels06}$ and Proposition $\ref{prop:weakDN_re}$,
and set $\om_0=\min(\om_1,\om_2)$.
Then \eqref{eq:resolvent} is equivalent to \eqref{eq:reduced}
for every $\la \in \Si_\ep \cup \{\om\in\BC \mid |\om|<\om_0\}$ and $\Bf \in J_q(\Om)$,
which means that the following assertions hold true:
$(\Bv,\Fq)=(\Bv,K(\Bv)) \in W_q^2(\Om)^N \times W_q^1(\Om)$ is a unique solution to \eqref{eq:resolvent}
if $\Bv \in W_q^2(\Om)^N$ is a solution to \eqref{eq:reduced},
and conversely, $\Bv \in W_q^2(\Om)^N$ is a unique solution to \eqref{eq:reduced}
if $(\Bv,\Fq) \in W_q^2(\Om)^N\times W_q^1(\Om)$ is a solution to \eqref{eq:resolvent}.
\end{prop}


\begin{proof}
Suppose that $\Bv \in W_q^2(\Om)^N$ is a solution to \eqref{eq:reduced}.
Let $\ph\in W_{q',\Ga}^1(\Om)$ with $1/q + 1/q' = 1$.
We then see, by the definition of $K$, that $\di\Bv= 0$ on $\Ga$ and
\begin{align*}
0 &= -(\Bf,\nabla\ph)_\Om
= -(\la\Bv-\Di\BT(\Bv,K(\Bv)),\nabla\ph)_\Om
= -(\la\Bv-\nabla\di\Bv,\nabla\ph)_\Om \\
&= (\la\di\Bv,\ph)_\Om+(\nabla\di\Bv,\nabla\ph)_\Om.
\end{align*}
Hence, $\di\Bv=0$ by Proposition \ref{prop:weakDN_re} when
$\la\in \Si_\ep \cup \{\om \in \BC \mid |\om|<\om_0\}$,
and therefore setting $\Fq=K(\Bv)$ implies that $(\Bv,\Fq)\in W_q^2(\Om)^N\times W_q^1(\Om)$ solves \eqref{eq:resolvent}.
The uniqueness of solutions to \eqref{eq:resolvent} follows from Lemma \ref{lemm:Abels06}.

Next we show the opposite direction.
Suppose that $(\Bv,\Fq) \in W_q^2(\Om)^N\times W_q^1(\Om)$ is a solution to \eqref{eq:resolvent}.
Let $\ph\in W_{q',\Ga}^1(\Om)$, and then we see, by the definition of $K$ and $\di\Bv=0$ in $\Om$, that
\begin{align*}
0 &= (\Bf,\nabla\ph)_\Om = (\nabla\Fq-\Di(\mu\BD(\Bv)),\nabla\ph)_\Om
= (\nabla(\Fq-K(\Bv)),\nabla\ph)_\Om, \\
0 &= \Fq-K(\Bv) \quad \text{on $\Ga$},
\end{align*}
where we have used $(\la\Bv,\nabla\ph)_\Om=0$ by $\di\Bv=0$ in $\Om$ and $\Bv=0$ on $S$.
Combining these two equations with the uniqueness of Proposition \ref{prop:weakDN} implies $\Fq=K(\Bv)$.
Thus $\Bv$ is a solution to \eqref{eq:reduced}. 
The uniqueness of solutions to \eqref{eq:reduced} follows from the first half of this proof,
which completes the proof of the proposition.  
\end{proof}

In view of \eqref{eq:reduced}, we set the Stokes operator $A_q$
as $A_q\Bv = \Di\BT(\Bv,K(\Bv))$ with the domain $D(A_q)$:
\begin{align}\label{domain}
&D(A_q) = W_{q,\CB}^2(\Om)\cap J_q(\Om), \\
&W_{q,\CB}^2(\Om)=
\{\Bv \in W_q^2(\Om)^N \mid (\mu\BD(\Bv)\Be_N)_\tau=0 \text{ on $\Ga$,} \quad \Bv = 0 \text{ on $S$}\}.
\notag
\end{align}
The system \eqref{eq:reduced} then can be written as $\la\Bv-A_q\Bv=\Bf$,
and one has

\begin{lemm}\label{lemm:Aq}
Let $1<q<\infty$ and $0<\ep<\pi/2$. 
Suppose that $\om_0$ is the same positive constant as in Proposition $\ref{prop:equiv}$.
Then there exists a positive constant $C=C_{N,d,q,\ep,\mu,\om_1,\om_2}$ such that,
for every $\la\in\Si_\ep \cup \{\om\in\BC \mid |\om|<\om_0\}$,
\begin{equation*}
\|(\la - A_q)^{-1}\|_{\CL(J_q(\Om))} \leq \frac{C}{1+|\la|}.
\end{equation*}
In addition, $A_q$ is a densely defined closed operator on $J_q(\Om)$. 
\end{lemm}

\begin{proof}
The required estimate follows from Lemma \ref{lemm:Abels06} and Proposition \ref{prop:equiv}.

Let $\Bv\in D(A_q)$. Then $(\di\Bv,\ph)_\Om=-(\Bv,\nabla\ph)_\Om =0$ for any $\ph\in C_0^\infty(\Om)$,
because $\Bv\in J_q(\Om)$ and $C_0^\infty(\Om)\subset W_{q',\Ga}^1(\Om)$ with $q'=q/(q-1)$.
This implies $\di\Bv=0$ in $\Om$.
It thus holds by the definition of $K$ that
\begin{equation*}
(A_q\Bv,\nabla\ph)_\Om=(\nabla\di\Bv,\nabla\ph)_\Om=0 \quad \text{for all $\ph\in W_{q',\Ga}^1(\Om)$,}
\end{equation*}
which furnishes $A_q\Bv\in J_q(\Om)$.
Then, following the proof of \cite[Lemma 3.7]{SS08},
we can show the last assertion of Lemma \ref{lemm:Aq}.
This completes the proof of the lemma.
\end{proof}

The following proposition follows from Lemma \ref{lemm:Aq}
and the standard theory of analytic $C_0$-semigroups.

\begin{prop}\label{prop:sg}
Let $1<q<\infty$. 
Then $A_q$ generates an analytic $C_0$-semigroup $\{e^{A_q t}\}_{t \geq 0}$ on $J_q(\Om)$.
Furthermore, there exist positive constants $\si_0$ and $C=C_{N,d,q,\mu,\si_0}$ such that for any $t>0$
\begin{alignat*}{2}
\|e^{A_q t}\Ba\|_{J_q(\Om)} &\leq C e^{-2\si_0 t}\|\Ba\|_{J_q(\Om)} && \quad (\Ba\in J_q(\Om)), \\
\|\pa_t e^{A_q t}\Ba\|_{J_q(\Om)} &\leq C t^{-1}e^{-2\si_0 t}\|\Ba\|_{J_q(\Om)} && \quad (\Ba\in J_q(\Om)), \\
\|\pa_t e^{A_q t}\Ba\|_{J_q(\Om)} &\leq C e^{-2\si_0 t}\|\Ba\|_{D(A_q)} && \quad (\Ba\in D(A_q)), 
\end{alignat*}
where 
$\|\Ba\|_{D(A_q)}=\|\Ba\|_{J_q(\Om)}+\|A_q\Ba\|_{J_q(\Om)}$. 
\end{prop}

Recall $D_{q,p}(\Om)=(J_q(\Om),D(A_q))_{1-1/p,p}$ for $1<p,q<\infty$.
We have a corollary of Proposition \ref{prop:sg}
that can be proved in the same manner as in \cite[Theorem 3.9]{SS08}.

\begin{coro}\label{coro:sg}
Let $1<p,q<\infty$ and
$\si_0$ be the same positive constant as in Proposition $\ref{prop:sg}$.
Then, for every $\Ba\in D_{q,p}(\Om)$,
$(\Bu,\Fp)=(e^{A_q t}\Ba,K(e^{A_q t}\Ba))$ 
is a unique solution to \eqref{eq:sg}, and also 
\begin{equation*}
\|e^{\si_0 t}(\pa_t\Bu,\Bu,\nabla\Bu,\nabla^2\Bu)\|_{L_p(\BR_+,L_q(\Om))}
+\|e^{\si_0 t}\Fp\|_{L_p(\BR_+,W_q^1(\Om))} \leq C\|\Ba\|_{D_{q,p}(\Om)}
\end{equation*}
for a positive constant $C=C_{N,d,p,q,\mu,\si_0}$.
\end{coro}

\begin{rema}\label{rema:sg}
By \cite[Propositions 2.2.2, 2.2.8]{Lunardi95},
it holds  that $\lim_{t \to 0+}\|e^{A_q t}\Ba-\Ba\|_{D_{q,p}(\Om)}=0$ for any $\Ba\in D_{q,p}(\Om)$.
Then, under the conditions $2/p+1/q\neq 1$ and $2/p+1/q\neq 2$,
we observe by Remark \ref{rema:domain} that $\lim_{t\to0+}\|e^{A_q t}\Ba-\Ba\|_{B_{q,p}^{2-2/p}(\Om)}=0$ for any $\Ba\in D_{q,p}(\Om)$. 
\end{rema}

\section{Proof of Theorem \ref{theo:main2}}\label{sec:maxLpLq}
This section proves Theorem \ref{theo:main2}.
Assume that $\si_0$ is the same positive constant as in Proposition \ref{prop:sg} in what follows.

{\bf Step 1.} 
The aim of this step is to decompose $(\Bu,\Fp)$ of \eqref{eq:linear}.
Let
\begin{equation*}
\Bu = \Bu_1 + \Bu_2 +\wt{\Bu}, \quad
\Fp = \Fp_1 + \Fp_2 + \wt{\Fp},
\end{equation*}
where each term on the right-hand sides satisfies the following systems:
\begin{align}
\label{eq:divide1}
&\left\{\begin{aligned} 
\pa_t\Bu_1-\Di\BT(\Bu_1,\Fp_1) &= 0 && \text{in $\Om$, $t>0$,}  \\
\di\Bu_1&=0 && \text{in $\Om$, $t>0$,} \\
\BT(\Bu_1,\Fp_1)\Be_N &= 0 && \text{on $\Ga$, $t>0$,} \\
\Bu_1 &= 0 && \text{on $S$, $t>0$,} \\
\Bu_1|_{t=0} &= \Ba && \text{in $\Om$,}
\end{aligned}\right. \\
\label{eq:divide2}
&\left\{\begin{aligned}
\pa_t\Bu_2+2\si_0\Bu_2-\Di\BT(\Bu_2,\Fp_2) &= 0 && \text{in $\Om$, $t>0$,} \\
\di\Bu_2&=g && \text{in $\Om$, $t>0$,} \\
\BT(\Bu_2,\Fp_2)\Be_N &= \Bh && \text{on $\Ga$, $t>0$,} \\
\Bu_2 &= 0 && \text{on $S$, $t>0$,} \\
\Bu_2|_{t=0} &= 0 && \text{in $\Om$,}
\end{aligned}\right. \\
\label{eq:dividetil}
&\left\{\begin{aligned}
\pa_t\wt{\Bu}-\Di\BT(\wt{\Bu},\wt{\Fp}) &= \Bf+2\si_0\Bu_2 && \text{in $\Om$, $t>0$,} \\
\di\wt{\Bu} &=0 && \text{in $\Om$, $t>0$,} \\
\BT(\wt{\Bu},\wt{\Fp})\Be_N &= 0 && \text{on $\Ga$, $t>0$,} \\
\wt{\Bu} &= 0 && \text{on $S$, $t>0$,} \\
\wt{\Bu}|_{t=0} &= 0 && \text{in $\Om$.}
\notag
\end{aligned}\right.
\end{align}

Let $P_q$ and $Q_q$ be the operators studied in Subsection \ref{subsec:HD}.
We then have
\begin{equation*}
\Bf+2\si_0\Bu_2 = P_q(\Bf+2\si_0\Bu_2) + \nabla Q_q(\Bf+2\si_0\Bu_2),
\end{equation*}
which gives further decompositions of $(\wt\Bu,\wt\Fp)$ 
as follows:
\begin{equation*}
\wt{\Bu}=\Bu_3+\Bu_4,\quad \wt{\Fp}=\Fp_3+\Fp_4,
\end{equation*}
where 
\begin{align}
\label{eq:divide3}
&\left\{\begin{aligned}
\pa_t\Bu_3+2\si_0\Bu_3-\Di\BT(\Bu_3,\Fp_3) &= \nabla Q_q(\Bf+2\si_0\Bu_2) && \text{in $\Om$, $t>0$,} \\
\di\Bu_3 &= 0 && \text{in $\Om$, $t>0$,} \\
\BT(\Bu_3,\Fp_3)\Be_N &= 0 && \text{on $\Ga$, $t>0$,} \\
\Bu_3 &= 0 && \text{on $S$, $t>0$,} \\
\Bu_3|_{t=0} &= 0 && \text{in $\Om$,}  
\end{aligned}\right. \\
\label{eq:divide4}
&\left\{\begin{aligned}
\pa_t\Bu_4 -\Di\BT(\Bu_4,\Fp_4) &= P_q(\Bf+2\si_0\Bu_2)+2\si_0\Bu_3 && \text{in $\Om$, $t>0$,} \\
\di\Bu_4 &= 0 && \text{in $\Om$, $t>0$,} \\ 
\BT(\Bu_4,\Fp_4)\Be_N &= 0 && \text{on $\Ga$, $t>0$,} \\
\Bu_4 &= 0 && \text{on $S$, $t>0$,} \\
\Bu_4|_{t=0} &= 0 && \text{in $\Om$}.
\end{aligned}\right.
\end{align}

By Corollary \ref{coro:sg} and Remark \ref{rema:sg},
there exists a solution $(\Bu_1,\Fp_1)$ of \eqref{eq:divide1}
such that $\lim_{t\to0+}\|\Bu_1(t)-\Ba\|_{B_{q,p}^{2-2/p}(\Om)}=0$ and
\begin{align}\label{170224_10}
&\|e^{\si_0 t}(\pa_t\Bu_1,\Bu_1,\nabla\Bu_1,\nabla^2\Bu_1)\|_{L_p(\BR_+,L_q(\Om))}
+\|e^{\si_0 t}\Fp_1\|_{L_p(\BR_+,W_q^1(\Om))} \\
&\leq C_{N,d,p,q,\mu,\si_0}\|\Ba\|_{D_{q,p}(\Om)}. \notag
\end{align}

{\bf Step 2.}
We consider \eqref{eq:divide2} in this step.
To this end, one extends $g$ and $\Bh$  to functions defined on
the whole line with respect to time $t$.
Let $E_0$ be the zero extension operator and $X$ be a Banach space. 
Then,
\begin{equation}\label{E0:1}
E_0\in \CL(L_p(\BR_+,X), L_p(\BR,X))\cap \CL({}_0W_p^1(\BR_+,X),W_p^1(\BR,X)).
\end{equation}
Especially, by \eqref{E0:1} with $X=L_q(\Om)$ and the complex interpolation method, 
\begin{equation}\label{E0:2}
E_0 \in \CL({}_0 H_p^{1/2}(\BR_+,L_q(\Om)),H_p^{1/2}(\BR,L_q(\Om))),
\end{equation}
where we have used the fact that
\begin{equation*}
[L_p(\BR,L_q(\Om)),W_p^1(\BR,L_q(\Om))]_{1/2}=H_p^{1/2}(\BR,L_q(\Om))
\end{equation*}
(cf. e.g. \cite[Theorem 1.56]{DK13}).
By \eqref{E0:1} and \eqref{E0:2}, one has
\begin{equation*}
E_0(e^{\si_0 t}g) \in H_{q,p}^{1,1/2}(\Om\times\BR)\cap W_p^1(\BR,\wh W_{q,\Ga}^{-1}(\Om)), \quad 
E_0(e^{\si_ 0 t}\Bh) \in H_{q,p}^{1,1/2}(\Om\times\BR)^N, \notag
\end{equation*}
together with the estimates:
\begin{align*}
\|E_0(e^{\si_0 t}g)\|_{H_{q,p}^{1,1/2}(\Om\times\BR)} &\leq C_{p,q}\|e^{\si_0 t}g\|_{{}_0H_{q,p}^{1,1/2}(\Om\times\BR_+)}, \\
\|E_0(e^{\si_0 t}g)\|_{W_p^1(\BR,\wh W_{q,\Ga}^{-1}(\Om))} &\leq C_{p,q}\|e^{\si_0 t}g\|_{W_p^1(\BR_+,\wh W_{q,\Ga}^{-1}(\Om))}, \notag \\
\|E_0(e^{\si_0 t}\Bh)\|_{H_{q,p}^{1,1/2}(\Om\times\BR)} &\leq C_{p,q}\|e^{\si_0 t}\Bh\|_{{}_0H_{q,p}^{1,1/2}(\Om\times\BR_+)}. \notag
\end{align*}
Thus, setting $G=e^{-\si_0 t}E_0(e^{\si_0 t}g)$ and $\BH=e^{-\si_0 t}E_0(e^{\si_0 t}\Bh)$ yields
\begin{align*}
&G\in H_{q,p,-\si_0}^{1,1/2}(\Om\times\BR) \cap W_{p,-\si_0}^1(\BR,\wh W_{q,\Ga}^{-1}(\Om)), \quad
\BH\in H_{q,p,-\si_0}^{1,1/2}(\Om\times\BR)^N, \\
&G=
\left\{\begin{aligned}
&g &&  (t>0), \\
&0 &&  (t<0),
\end{aligned}\right. \quad
\BH=
\left\{\begin{aligned}
&\Bh && (t>0), \\
&0 && (t<0), 
\end{aligned}\right.
\end{align*}
and also
\begin{align*}
\|e^{\si_0 t}G\|_{H_{q,p}^{1,1/2}(\Om\times\BR)} &\leq C_{p,q}\|e^{\si_0 t}g\|_{{}_0H_{q,p}^{1,1/2}(\Om\times\BR_+)}, \\
\|e^{\si_0 t}(\pa_ tG, G)\|_{L_p(\BR,\wh {W}_{q,\Ga}^{-1}(\Om))} 
&\leq C_{p,q}\|e^{\si_0 t}g\|_{W_p^1(\BR_+,\wh W_{q,\Ga}^{-1}(\Om))}, \\
\|e^{\si_0 t}\BH\|_{H_{q,p}^{1,1/2}(\Om\times\BR)} &\leq C_{p,q}\|e^{\si_0 t}\Bh\|_{{}_0H_{q,p}^{1,1/2}(\Om\times\BR_+)}.
\end{align*}
Combining these properties with Proposition \ref{prop:shifted} for $\de=\si_0$
furnishes that there is a solution $(\BU,P)\in W_{q,p,-\si_0}^{2,1}(\Om\times\BR)^N\times L_{p,-\si_0}(\BR,W_q^1(\Om))$ to
\begin{equation*}
\left\{\begin{aligned}
\pa_t\BU +2\si_0\BU-\Di\BT(\BU,P)&=0 && \text{in $\Om$, $t\in\BR$,} \\
\di\BU&=G &&  \text{in $\Om$, $t\in\BR$,} \\
\BT(\BU,P)\Be_N &=\BH && \text{on $\Ga$, $t\in\BR$,} \\
\BU&= 0 && \text{on $S$, $t\in\BR$,}
\end{aligned}\right.
\end{equation*}
and furthermore, $\BU$ vanishes for $t<0$ and
\begin{align*}
&\|e^{\si_0 t}(\pa_t\BU,\BU,\nabla\BU,\nabla^2\BU)\|_{L_p(\BR,L_q(\Om))}
+\|e^{\si_0 t}P\|_{L_p(\BR,W_q^1(\Om))}\\
&\leq C_{N,d,p,q,\mu,\si_0}\left(\|e^{\si_0 t}g\|_{W_p^1(\BR_+,\wh W_{q,\Ga}^{-1}(\Om))}
+\|e^{\si_0 t}(g,\Bh)\|_{{}_0H_{q,p}^{1,1/2}(\Om\times\BR_+)}\right).
\end{align*}
Therefore, $(\Bu_2,\Fp_2)=(\BU,P)$ solves \eqref{eq:divide2}
and satisfies 
\begin{align}\label{170225_2}
&\|e^{\si_0 t}(\pa_t \Bu_2,\Bu_2,\nabla\Bu_2,\nabla^2\Bu_2)\|_{L_p(\BR_+,L_q(\Om))}
+\|e^{\si_0 t}\Fp_2\|_{L_p(\BR_+,L_q(\Om))} \\
&\leq C_{N,d,p,q,\mu,\si_0}\left(\|e^{\si_0 t}g\|_{W_p^1(\BR_+,\wh W_{q,\Ga}^{-1}(\Om))}
+\|e^{\si_0 t}(g,\Bh)\|_{{}_0H_{q,p}^{1,1/2}(\Om\times\BR_+)}\right). \notag
\end{align}

{\bf Step 3.} 
We consider \eqref{eq:divide3} in this step.
By Proposition \ref{prop:decomp} and \eqref{170225_2},
\begin{align*}
&\|e^{\si_0 t}E_0\nabla Q_q(\Bf+2\si_0\Bu_2)\|_{L_p(\BR,L_q(\Om))} 
\leq C\|e^{\si_0 t}(\Bf+2\si_0\Bu_2)\|_{L_p(\BR_+,L_q(\Om))} \\
&\leq C_{N,d,p,q,\mu,\si_0}\Big(\|e^{\si_0 t}\Bf\|_{L_p(\BR_+,L_q(\Om))}+\|e^{\si_0 t}g\|_{W_p^1(\BR_+,\wh W_{q,\Ga}^{-1}(\Om))}\\
&\quad+\|e^{\si_0 t}(g,\Bh)\|_{{}_0H_{q,p}^{1,1/2}(\Om\times\BR_+)}\Big),
\end{align*}
which implies $E_0\nabla Q_q(\Bf+2\si_0\Bu_2)\in L_{p,-\si_0}(\BR,L_q(\Om))^N$.
One thus observes by Proposition \ref{prop:shifted} with $\de=\si_0$ that
there is a solution $(\BV,Q)\in W_{q,p,-\si_0}^{2,1}(\Om\times\BR)^N\times L_{p,-\si_0}(\BR,W_q^1(\Om))$ to
\begin{equation*}
\left\{\begin{aligned}
\pa_t\BV +2\si_0\BV-\Di\BT(\BV,Q)&=E_0\nabla Q_q(\Bf+2\si_0\Bu_2) && \text{in $\Om$, $t\in\BR$,} \\
\di\BV&=0 &&  \text{in $\Om$, $t\in\BR$,} \\
\BT(\BV,Q)\Be_N &=0 && \text{on $\Ga$, $t\in\BR$,} \\
\BV&= 0 && \text{on $S$, $t\in\BR$,}
\end{aligned}\right.
\end{equation*}
and furthermore, $\BV$ vanishes for $t<0$ and
\begin{align*}
&\|e^{\si_0 t}(\pa_t\BV,\BV,\nabla\BV,\nabla^2\BV)\|_{L_p(\BR,L_q(\Om))}
+\|e^{\si_0 t}Q\|_{L_p(\BR,W_q^1(\Om))}\\
&\leq C_{N,d,p,q,\mu,\si_0}\Big(\|e^{\si_0 t}\Bf\|_{L_p(\BR_+,L_q(\Om))}
+\|e^{\si_0 t}g\|_{W_p^1(\BR_+,\wh W_{q,\Ga}^{-1}(\Om))}\\
&\quad +\|e^{\si_0 t}(g,\Bh)\|_{{}_0H_{q,p}^{1,1/2}(\Om\times\BR_+)}\Big).
\end{align*}
It is then clear that $(\Bu_3,\Fp_3)=(\BV,Q)$ solves \eqref{eq:divide3} and satisfies
\begin{align}\label{est:divide3}
&\|e^{\si_0 t}(\pa_t\Bu_3,\Bu_3,\nabla\Bu_3,\nabla^2\Bu_3)\|_{L_p(\BR_+,L_q(\Om))}
+\|e^{\si_0 t}\Fp_3\|_{L_p(\BR_+,W_q^1(\Om))} \\
&\leq C_{N,d,p,q,\mu,\si_0}\Big(\|e^{\si_0 t}\Bf\|_{L_p(\BR_+,L_q(\Om))}
+\|e^{\si_0 t}g\|_{W_p^1(\BR_+,\wh W_{q,\Ga}^{-1}(\Om))} \notag \\
&\quad+\|e^{\si_0 t}(g,\Bh)\|_{{}_0H_{q,p}^{1,1/2}(\Om\times\BR_+)}\Big). \notag
\end{align}

{\bf Step 4.} 
We consider \eqref{eq:divide4} in this step.
Let $\BF=P_q(\Bf+2\si_0\Bu_2)+2\si_0\Bu_3$.
Then, by Proposition \ref{prop:decomp}, \eqref{170225_2}, and \eqref{est:divide3},
one has $e^{\si_0 t}\BF\in L_p(\BR_+,J_q(\Om))$ with 
\begin{align}\label{170721_1}
&\|e^{\si_0 t}\BF\|_{L_p(\BR_+,L_q(\Om))}
\leq C_{N,d,p,q,\mu,\si_0}\Big(\|e^{\si_0 t}\Bf\|_{L_p(\BR_+,L_q(\Om))}
 \\
&\quad
+\|e^{\si_0 t}g\|_{W_p^1(\BR_+,\wh W_{q,\Ga}^{-1}(\Om))}
+\|e^{\si_0 t}(g,\Bh)\|_{{}_0H_{q,p}^{1,1/2}(\Om\times\BR_+)}\Big). \notag
\end{align}
Such an $\BF$ can be approximated by an element of $C_0^\infty(\BR_+,J_q(\Om))$
under the norm $\|e^{\si_0 t}\cdot\|_{L_p(\BR_+,L_q(\Om))}$,
so that it suffices to consider the case $\BF\in C_0^\infty(\BR_+,J_q(\Om))$. 

As the first step, we estimate the solution $\Bu_4$ to  \eqref{eq:divide4} given by Duhamel's formula:
\begin{equation*}
\Bu_4(t)= \int_0^t e^{A_q(t-s)}\BF(s)\intd s \quad (t>0),
\end{equation*}
which is also written as
\begin{equation*}
e^{\si_0 t}\Bu_4(t) = \int_0^t e^{\si_0(t-s)}e^{A_q(t-s)}(e^{\si_0 s}\BF(s))\intd s.
\end{equation*}
Let $\chi_A$ be the characteristic function of $A\subset\BR$, and then by Proposition \ref{prop:sg}
\begin{align*}
\|e^{\si_0 t} \Bu_4(t)\|_{L_q(\Om)}
&\leq C_{N,d,q,\mu,\si_0}\int_0^t e^{-\si_0(t-s)}\|e^{\si_0 s}\BF(s)\|_{L_q(\Om)}\intd s \\
&= C_{N,d,q,\mu,\si_0}\int_{-\infty}^t e^{-\si_0(t-s)}\|e^{\si_0 s}\BF(s)\|_{L_q(\Om)}\intd s \\
&= C_{N,d,q,\mu,\si_0}\int_{-\infty}^\infty\chi_{(0,\infty)}(t-s)e^{-\si_0(t-s)}\|e^{\si_0 s}\BF(s)\|_{L_q(\Om)}\intd s \\
&= C_{N,d,q,\mu,\si_0}(\chi_{(0,\infty)}(\cdot)e^{-\si_0 \cdot}*(e^{\si_0\cdot}\,\BF(\cdot)))(t),
\end{align*}
where $(f*g)(t)=\int_\BR f(t-s)g(s)\intd s$.
Combining this inequality with Young's inequality $\|f*g\|_{L_p(\BR)}\leq\|f\|_{L_1(\BR)}\|g\|_{L_p(\BR)}$ furnishes that
\begin{align}\label{est:u4}
\|e^{\si_0 t} \Bu_4\|_{L_p(\BR_+, L_q(\Om))}
&\leq C_{N,d,q,\mu,\si_0}\|\chi_{(0,\infty)}(\cdot)e^{-\si_0 \cdot}\|_{L_1(\BR)}\|e^{\si_0 t}\BF\|_{L_p(\BR,L_q(\Om))} \\
&\leq C_{N,d,q,\mu,\si_0}\|e^{\si_0 t}\BF\|_{L_p(\BR_+,L_q(\Om))}. \notag
\end{align}

As the second step,
we write \eqref{eq:divide4} as
\begin{equation*}
\left\{\begin{aligned}
\pa_t\Bu_4 +2\si_0\Bu_4 -\Di\BT(\Bu_4,\Fp_4) &= \BF + 2\si_0\Bu_4 && \text{in $\Om$, $t>0$,} \\
\di\Bu_4 &= 0 && \text{in $\Om$, $t>0$,} \\ 
\BT(\Bu_4,\te_4)\Be_N &= 0 && \text{on $\Ga$, $t>0$,} \\
\Bu_4 &= 0 && \text{on $S$, $t>0$,} \\
\Bu_4|_{t=0} &= 0 && \text{in $\Om$.}
\end{aligned}\right.
\end{equation*}
Then, similarly to Step 3, one observes by \eqref{est:u4} that
\begin{align*}
&\|e^{\si_0 t}(\pa_t\Bu_4,\Bu_4,\nabla\Bu_4,\nabla^2\Bu_4)\|_{L_p(\BR_+,L_q(\Om))}
+\|e^{\si_0 t}\Fp_4\|_{L_p(\BR_+,W_q^1(\Om))} \\
&\leq C_{N,d,p,q,\mu,\si_0}\|e^{\si_0 t}(\BF,\Bu_4)\|_{L_p(\BR_+,L_q(\Om))} \notag \\
&\leq C_{N,d,p,q,\mu,\si_0}\|e^{\si_0 t}\BF\|_{L_p(\BR_+,L_q(\Om))},
\end{align*}
which, combined with \eqref{170721_1}, furnishes
\begin{align}\label{est:divide4}
&\|e^{\si_0 t}(\pa_t\Bu_4,\Bu_4,\nabla\Bu_4,\nabla^2\Bu_4)\|_{L_p(\BR_+,L_q(\Om))}
+\|e^{\si_0 t}\Fp_4\|_{L_p(\BR_+,W_q^1(\Om))} \\
&\leq C_{N,d,p,q,\mu,\si_0}\Big(\|e^{\si_0 t}\Bf\|_{L_p(\BR_+,L_q(\Om))}
+\|e^{\si_0 t}g\|_{L_p(\BR_+,\wh W_{q,\Ga}^{-1}(\Om))} \notag \\
&\quad+\|e^{\si_0 t}(g,\Bh)\|_{{}_0H_{q,p}^{1,1/2}(\Om\times\BR_+)}\Big). \notag
\end{align}

{\bf Step 5.}
Let $\Fp=\sum_{j=1}^4\Fp_j$ and $\Bu=\Bu_1+\Bv$ for $\Bv=\sum_{j=2}^4\Bu_j$,
and then $(\Bu,\Fp)$ is a solution to \eqref{eq:linear}
and satisfies the required estimate of Theorem \ref{theo:main2} by \eqref{170224_10}, \eqref{170225_2}, \eqref{est:divide3}, and \eqref{est:divide4}.

Finally, we consider the initial condition and the uniqueness of solutions.
For two Banach spaces $X$, $Y$, the symbol $X\hookrightarrow Y$ means that
$X$ is continuously injected in $Y$, i.e. there is a positive constant $C$
such that $\|f\|_Y\leq C\|f\|_X$ for any $f\in X$.
It is then well-known by e.g. \cite[Theorem 4.10.2]{Amann95} that 
\begin{equation}\label{embed:Amann}
W_p^1(\BR_+,X_0)\cap L_p(\BR_+,X_1)\hookrightarrow BUC([0,\infty),(X_0,X_1)_{1-1/p,p})
\end{equation}
for Banach spaces $X_0$, $X_1$ satisfying the properties: $X_1\hookrightarrow X_0$ and $X_1$ is dense in $X_0$. 
One makes use of \eqref{embed:Amann} with $X_0=L_q(\Om)$ and $X_1=W_q^2(\Om)$
in order to obtain
\begin{equation}\label{170615_1}
W_{q,p}^{2,1}(\Om\times \BR_+)
\hookrightarrow BUC([0,\infty),B_{q,p}^{2-2/p}(\Om)),
\end{equation}
which implies that $\Bv|_{t=0}=0$ in $B_{q,p}^{2-2/p}(\Om)^N$.
We thus conclude $\lim_{t\to0+}\|\Bu(t)-\Ba\|_{B_{q,p}^{2-2/p}(\Om)}=0$.
The uniqueness of solutions to \eqref{eq:linear}
follows from the existence of solutions for a dual problem (cf. e.g. \cite[Subsection 7.2]{Saito15b}),
which completes the proof of Theorem \ref{theo:main2}.

\section{Proof of Theorem \ref{theo:main}}\label{sec:nonl}
This section proves Theorem \ref{theo:main}.
To this end, we set for $1<q<\infty$
\begin{equation*}
B_q(\Om\times\BR_+)
= W_\infty^1(\BR_+,L_\infty(\Om))\cap L_\infty(\BR_+, W_q^1(\Om))
\end{equation*}
and start with 
\begin{lemm}\label{lemm:H1/2}
Let $1<p,q<\infty$.
Then the following assertions hold true.
\begin{enumerate}[$(1)$]
\item\label{lemm:H1/2_1}
For any $f\in W_\infty^1(\BR_+,L_\infty(\Om))$
and $g\in H_p^{1/2}(\BR_+,L_q(\Om))$,
\begin{equation*}
\|fg\|_{H_p^{1/2}(\BR_+,L_q(\Om))} 
\leq \|f\|_{W_\infty^1(\BR_+,L_\infty(\Om))}\|g\|_{H_p^{1/2}(\BR_+,L_q(\Om))}.
\end{equation*}
\item\label{lemm:H1/2_3}
Let $q>N$. Then there exists a positive constant $C_{N,q}$ such that
for any $f,g\in B_q(\Om\times\BR_+)$ and $h\in H_{q,p}^{1,1/2}(\Om\times\BR_+)$,
\begin{align*}
\|fg\|_{B_q(\Om\times\BR_+)}
&\leq C_{N,q}\|f\|_{B_q(\Om\times\BR_+)}\|g\|_{B_q(\Om\times\BR_+)}, \\
\|fh\|_{H_{q,p}^{1,1/2}(\Om\times\BR_+)}
&\leq C_{N,q}\|f\|_{B_q(\Om\times\BR_+)}\|h\|_{H_{q,p}^{1,1/2}(\Om\times\BR_+)}.
\end{align*}
\item\label{lemm:H1/2_2}
For any $f\in H_p^{1/2}(\BR_+,W_q^1(\Om))$,
\begin{equation*}
\|\nabla f\|_{H_p^{1/2}(\BR_+,L_q(\Om))} \leq \|f\|_{H_p^{1/2}(\BR_+,W_q^1(\Om))}.
\end{equation*}
\end{enumerate}
\end{lemm}

\begin{proof}
\eqref{lemm:H1/2_1}.
It clearly holds that
\begin{align*}
\|fg\|_{L_p(\BR_+,L_q(\Om))} 
&\leq \|f\|_{W_\infty^1(\BR_+,L_\infty(\Om))}\|g\|_{L_p(\BR_+,L_q(\Om))},\\
\|fg\|_{W_p^1(\BR_+,L_q(\Om))}
&\leq \|f\|_{W_\infty^1(\BR_+,L_\infty(\Om))}\|g\|_{W_p^1(\BR_+,L_q(\Om))}.
\end{align*}
Let $T_f g=fg$, and then $T_f \in \CL(L_p(\BR_+,L_q(\Om))\cap \CL(W_p^1(\BR_+,L_q(\Om))$ with
\begin{align*}
\|T_f\|_{\CL(L_p(\BR_+,L_q(\Om)))}
&\leq \|f\|_{W_\infty^1(\BR_+,L_\infty(\Om))}, \\
\|T_f\|_{\CL(W_p^1(\BR_+,L_q(\Om)))}
& \leq \|f\|_{W_\infty^1(\BR_+,L_\infty(\Om))}.
\end{align*}
Combining these properties with the complex interpolation method 
furnishes that
\begin{align*}
&\|T_f g\|_{[L_p(\BR_+,L_q(\Om),W_p^1(\BR_+,L_q(\Om)))]_{1/2}} \\
&\leq \|f\|_{W_\infty^1(\BR_+,L_\infty(\Om))}\|g\|_{[L_p(\BR_+,L_q(\Om),W_p^1(\BR_+,L_q(\Om)))]_{1/2}}.
\end{align*}
Noting 
$[L_p(\BR_+,L_q(\Om),W_p^1(\BR_+,L_q(\Om)))]_{1/2}=H_p^{1/2}(\BR_+,L_q(\Om))$ by the definition,
we complete the proof of Lemma \ref{lemm:H1/2} \eqref{lemm:H1/2_1}.

\eqref{lemm:H1/2_3}.
Since $W_q^1(\Om)$ is a Banach algebra for $q>N$,
we have
\begin{align}\label{170721_5}
\|fg\|_{L_\infty(\BR_+,W_q^1(\Om))}
&\leq C_{N,q}\|f\|_{L_\infty(\BR_+,W_q^1(\Om))}\|g\|_{L_\infty(\BR_+,W_q^1(\Om))}, \\
\|fh\|_{L_p(\BR_+,W_q^1(\Om))}
&\leq C_{N,q}\|f\|_{L_\infty(\BR_+,W_q^1(\Om))}\|h\|_{L_p(\BR_+,W_q^1(\Om))}. \notag
\end{align}
It is clear that $\|fg\|_{W_\infty^1(\BR_+,L_\infty(\Om))}\leq \|f\|_{W_\infty^1(\BR_+,L_\infty(\Om))}\|g\|_{W_\infty^1(\BR_+,L_\infty(\Om))}$,
so that one has, together with the first inequality of \eqref{170721_5},
\begin{align*}
\|fg\|_{B_q(\Om\times\BR_+)}\leq C_{N,q}\|f\|_{B_q(\Om\times\BR_+)}\|g\|_{B_q(\Om\times\BR_+)}.
\end{align*}
The other estimate of Lemma \ref{lemm:H1/2} \eqref{lemm:H1/2_3} follows from 
Lemma \ref{lemm:H1/2} (1) and the second inequality of \eqref{170721_5}.


\eqref{lemm:H1/2_2}.
The required inequality follows from the complex interpolation method immediately,
which completes the proof of Lemma \ref{lemm:H1/2} \eqref{lemm:H1/2_2}.
\end{proof}


At this point, we introduce some embedding properties.

\begin{lemm}\label{lemm:embed}
Let $p$, $q$ satisfy \eqref{pq}.
Then there exists a positive constant $M_1\geq 1$ such that the following assertions hold true.
\begin{enumerate}[$(1)$]
\item\label{lemm:embed_1}
$\|f\|_{BUC(\overline\Om)}\leq M_1\|f\|_{W_q^1(\Om)}$ for any $f\in W_q^1(\Om)$.
\item\label{lemm:embed_2}
$\|f\|_{BUC([0,\infty),BUC^1(\overline{\Om}))}\leq M_1\|f\|_{W_{q,p}^{2,1}(\Om\times\BR_+)}$
for any $f\in W_{q,p}^{2,1}(\Om\times\BR_+)$.
\item\label{lemm:embed_3}
$\|f\|_{H_p^{1/2}(\BR_+,W_q^1(\Om))}\leq M_1\|f\|_{W_{q,p}^{2,1}(\Om\times\BR_+)}$
for any $f\in W_{q,p}^{2,1}(\Om\times\BR_+)$.
\end{enumerate}
\end{lemm}

\begin{proof}
\eqref{lemm:embed_1}.
See e.g. \cite[Theorem 4.12]{AF03}.

\eqref{lemm:embed_2}.
It follows from \eqref{170615_1} and
$B_{q,p}^{2-2/p}(\Om)\hookrightarrow BUC^1(\overline{\Om})$ under the condition \eqref{pq}
(cf. \cite[Theorem 4.6.1]{Triebel78}).

\eqref{lemm:embed_3}.
See e.g. \cite[Proposition 3.2, Remark 3.3]{MS12}, \cite{Shibata17pre}.
\end{proof}

%
%
%
%
%
%

We prove Theorem \ref{theo:main} in the remaining part of this section
by using Theorem \ref{theo:main2} and the contraction mapping theorem.
Let $\ga_0=\si_0$, where $\si_0$ is the same positive constant as in Theorem \ref{theo:main2}.
We here define a closed set $X_{q,p}^{\ga_0}(R)$ by
\begin{align*}
X_{q,p}^{\ga_0}(R)&=
\{(\Bu,\Fp)\in W_{q,p}^{2,1}(\Om\times\BR_+)^N\times L_p(\BR_+,W_q^1(\Om))
\mid \|(\Bu,\Fp)\|_{X_{q,p}^{\ga_0}}\leq R, \\
&\qquad\Bu=0 \text{ on $S$, }
\lim_{t\to 0+}\|\Bu-\Ba\|_{B_{q,p}^{2-2/p}(\Om)}=0\}, 
\end{align*}
where $0<R<1$ is a positive constant  and
$\|(\Bu,\Fp)\|_{X_{q,p}^{\ga_0}}= \|\Bu\|_{Y_{q,p}^{\ga_0}}+\|\Fp\|_{Z_{q,p}^{\ga_0}}$ with 
\begin{equation*}
\|\Bu\|_{Y_{q,p}^{\ga_0}}=\|e^{\ga_0 t}(\pa_t\Bu,\Bu,\nabla\Bu,\nabla^2\Bu)\|_{L_p(\BR_+,L_q(\Om))}, \quad 
\|\Fp\|_{Z_{q,p}^{\ga_0}}=\|e^{\ga_0 t}\Fp\|_{L_p(\BR_+,W_q^1(\Om))}.
\end{equation*}
In addition, by e.g. \cite{LMV17pre}, we know  the following characterization:
\begin{equation}\label{chara:1}
{}_0 H_{q,p}^{1,1/2}(\Om\times\BR_+)
=\{f\in H_{q,p}^{1,1/2}(\Om\times\BR_+) \mid f|_{t=0}=0 \text{ in $L_q(\Om)$}\}
\end{equation}
under the condition \eqref{pq}.


Let $(\Bv,\Fq)\in X_{q,p}^{\ga_0}(R)$, and one considers
\begin{equation}\label{170522_1}
\left\{\begin{aligned}
\pa_t\Bu-\Di\BT(\Bu,\Fp) = \BF(\Bv)& && \text{in $\Om$, $t>0$,} \\
\di\Bu = G(\Bu) = \di\BG(\Bv)& && \text{in $\Om$, $t>0$,} \\
\BT(\Bu,\Fp)\Be_N = \BH(\Bv)\Be_N& && \text{on $\Ga$, $t>0$,} \\
\Bu = 0& && \text{on $S$, $t>0$,} \\
\Bu|_{t=0} = \Ba& && \text{in $\Om$.}
\end{aligned}\right.
\end{equation}
Here we prove

\begin{lemm}\label{lemm:nonl}
Let $p$, $q$ satisfy \eqref{pq}, and let $\ga_0$ be as above.
Then there exists positive constants $M_3\geq M_2 \geq 1$ such that the following assertions hold true.
\begin{enumerate}[$(1)$]
\item\label{lemm:nonl_1}
For any $\Bv_r$ with $\|\Bv_r\|_{Y_{q,p}^{\ga_0}}\leq 1$ $(r=1,2)$, we have
\begin{align*}
&\left\|\BU_1\left(\int_0^t\nabla\Bv_2(\xi,s)\intd s\right)
-\BU_1\left(\int_0^t\nabla\Bv_1(\xi,s)\intd s\right)\right\|_{L_\infty(\BR_+,L_\infty(\Om))} \\
&\leq M_2\|\Bv_2-\Bv_1\|_{Y_{q,p}^{\ga_0}}, \\
&\sum_{i,j,k,l=1}^N\left\|V_{ijk}^l\left(\int_0^t\nabla\Bv_2(\xi,s)\intd s\right)
-V_{ijk}^l\left(\int_0^t\nabla\Bv_1(\xi,s)\intd s\right)\right\|_{L_\infty(\BR_+,L_\infty(\Om))} \\
&\leq M_2\|\Bv_2-\Bv_1\|_{Y_{q,p}^{\ga_0}}, \\
&\sum_{
\stackrel{\text{\tiny{$i,j,k,l,$}}}{m,n=1}
}^N\left\|W_{ijklm}^n\left(\int_0^t\nabla\Bv_2(\xi,s)\intd s\right)
-W_{ijklm}^n\left(\int_0^t\nabla\Bv_1(\xi,s)\intd s\right)\right\|_{L_\infty(\BR_+,L_\infty(\Om))} \\
&\leq M_2\|\Bv_2-\Bv_1\|_{Y_{q,p}^{\ga_0}}, \\
&\left\|\BU_r\left(\int_0^t\nabla\Bv_2(\xi,s)\intd s\right)
-\BU_r\left(\int_0^t\nabla\Bv_1(\xi,s)\intd s\right)\right\|_{B_q(\Om\times\BR_+)} \\
&\leq M_2\|\Bv_2-\Bv_1\|_{Y_{q,p}^{\ga_0}} \quad (r=2,\dots,7).
\end{align*}
\item\label{lemm:nonl_2}
For any $\Bv$ with $\|\Bv\|_{Y_{q,p}^{\ga_0}}\leq 1$, we have
\begin{align*}
\left\|\BU_1\left(\int_0^t\nabla\Bv(\xi,s)\intd s\right)\right\|_{L_\infty(\BR_+,L_\infty(\Om))}
&\leq M_2\|\Bv\|_{Y_{q,p}^{\ga_0}}, \\
\sum_{i,j,k,l=1}^N\left\|V_{ijk}^l\left(\int_0^t\nabla\Bv(\xi,s)\intd s\right)\right\|_{L_\infty(\BR_+,L_\infty(\Om))}
&\leq M_2\|\Bv\|_{Y_{q,p}^{\ga_0}}, \\
\left\|\BU_r\left(\int_0^t\nabla\Bv(\xi,s)\intd s\right)\right\|_{B_q(\Om\times\BR_+)}
&\leq M_2\|\Bv\|_{Y_{q,p}^{\ga_0}} \quad (r=2,\dots,7).
\end{align*}
\item\label{lemm:nonl_3}
For any $\Bv$ with $\|\Bv\|_{Y_{q,p}^{\ga_0}}\leq 1$, we have
\begin{equation*}
\sum_{i,j,k,l,m,n=1}^N\left\|W_{ijklm}^n\left(\int_0^t\nabla\Bv(\xi,s)\intd s\right)\right\|_{L_\infty(\BR_+,L_\infty(\Om))}
 \leq M_3, \\
\end{equation*}
\end{enumerate}
\end{lemm}

\begin{proof}
For $F(\BX)$ with $\BX=(X_{\al\beta})\in\BR^{N\times N}$, 
we denote the derivative of $F(\BX)$ with respect to $X_{\al\beta}$ by $\pa_{(\al,\beta)}F(\BX)$ 
for $(\al,\beta)\in\CN:=\{1,\dots,N\}\times \{1,\dots,N\}$.
In addition, we choose $m_2\geq 1$ large enough so that for $r_1=1,\dots,7$ and $r_2=2,\dots,7$
\begin{align*}
\max\left\{\sum_{(\al,\beta)\in\CN}\left|\pa_{(\al,\beta)}\BU_{r_1}\left(\BX\right)\right|\mid |\BX| \leq 2M_1(p'\ga_0)^{-1/p'}\right\} 
&\leq m_2, \\
\max\left\{\sum_{(\al',\beta')\in\CN}\sum_{(\al,\beta)\in\CN}\left|\pa_{(\al',\beta')}\pa_{(\al,\beta)}\BU_{r_2}
\left(\BX\right)\right|\mid |\BX| \leq 2M_1(p'\ga_0)^{-1/p'}\right\} 
&\leq m_2, \notag  \\
\max\left\{\sum_{(\al,\beta)\in\CN}\sum_{i,j,k,l=1}^N
\left|\pa_{(\al,\beta)}V_{ijk}^l\left(\BX\right)\right|\mid |\BX| \leq 2M_1(p'\ga_0)^{-1/p'}\right\} 
&\leq m_2, \notag \\
\max\left\{\sum_{(\al,\beta)\in\CN}\sum_{i,j,k,l,m,n=1}^N
\left|\pa_{(\al,\beta)}W_{ijklm}^n\left(\BX\right)\right|\mid |\BX| \leq 2M_1(p'\ga_0)^{-1/p'}\right\} 
&\leq m_2, \notag 
\end{align*}
where $p'=p/(p-1)$ and $M_1$ is the same positive constant as in Lemma \ref{lemm:embed}.

\eqref{lemm:nonl_1}.
First, we consider $\BU_1$.
Let us write $\Bv_r=(v_{r1},\dots,v_{rN})^\SST$ for $r=1,2$.
It then holds that
\begin{align}\label{170709_7}
&\BU_1\left(\int_0^t\nabla\Bv_2(\xi,s)\intd s\right)
-\BU_1\left(\int_0^t\nabla\Bv_1(\xi,s)\intd s\right) \\
&=
\sum_{(\al,\beta)\in\CN}\int_0^1(\pa_{(\al,\beta)}\BU_1)
\left(\te\int_0^t\nabla\Bv_2(\xi,s)\intd s+(1-\te)\int_0^t\nabla\Bv_1(\xi,s)\intd s\right)\intd\te \notag \\
&\quad \cdot\left(\int_0^t\pa_\beta v_{2\al}(\xi,s)\intd s-\int_0^t\pa_\beta v_{1\al}(\xi,s)\intd s\right) \notag
\end{align}
and that by Lemma \ref{lemm:embed} \eqref{lemm:embed_1}
\begin{align}\label{170709_1}
&\left\|\te\int_0^t\nabla\Bv_2(\xi,s)\intd s+(1-\te)\int_0^t\nabla\Bv_1(\xi,s)\intd s\right\|_{L_\infty(\BR_+,L_\infty(\Om))}\\
&\leq \sum_{r=1}^2\int_0^\infty\|\nabla\Bv_r(\cdot,s)\|_{L_\infty(\Om)}\intd s
\leq M_1\sum_{r=1}^2\int_0^\infty\|\Bv_r(\cdot,s)\|_{W_q^2(\Om)}\intd s \notag \\
&\leq M_1\sum_{r=1}^2\left(\int_0^\infty e^{-p'\ga_0 s}\intd s\right)^{1/p'}\|\Bv_r\|_{Y_{q,p}^{\ga_0}}
\leq 2M_1(p'\ga_0)^{-1/p'}. \notag
\end{align}
One similarly obtains 
\begin{align}
&\left\|\int_0^t\nabla \By(\xi,s)\intd s \right\|_{L_\infty(\BR_+,L_\infty(\Om))}
\leq M_1(p'\ga_0)^{-1/p'}\|\By\|_{Y_{q,p}^{\ga_0}}, \label{170709_3} \\
&\left\|\int_0^t\nabla\By(\xi,s)\intd s\right\|_{L_\infty(\BR_+,W_q^1(\Om))} 
\leq (p'\ga_0)^{-1/p'}\|\By\|_{Y_{q,p}^{\ga_0}}. \label{170709_4}
\end{align} 
Combining \eqref{170709_7} with \eqref{170709_1} and \eqref{170709_3} furnishes that
\begin{align*}
&\left\|\BU_1\left(\int_0^t\nabla\Bv_2(\xi,s)\intd s\right)
-\BU_1\left(\int_0^t\nabla\Bv_1(\xi,s)\intd s\right)\right\|_{L_\infty(\BR_+,L_\infty(\Om))} \\
&\leq m_2M_1(p'\ga_0)^{-1/p'}\|\Bv_2-\Bv_1\|_{Y_{q,p}^{\ga_0}}.
\end{align*}
Thus we choose $M_2\geq 1$ large enough so that $M_2\geq m_2M_1(p'\ga_0)^{-1/p'}$
in order to obtain the required estimate for $\BU_1$.
Analogously, one can prove that the required estimates hold true for $V_{ijk}^l$ and $W_{ijklm}^n$.

Next, we consider $\BU_r$ for $r=2,\dots,7$.
Following \eqref{170709_7}, we observe that for $\pa\in\{\pa_t,\pa_1,\dots,\pa_N\}$ 
\begin{align*}
&\pa\left\{\BU_r\left(\int_0^t\nabla\Bv_2(\xi,s)\intd s\right)
-\BU_r\left(\int_0^t\nabla\Bv_1(\xi,s)\intd s\right)\right\} \\
&=\sum_{(\al',\beta')\in\CN}\sum_{(\al,\beta)\in\CN} \notag \\
& \int_0^1(\pa_{(\al',\beta')}\pa_{(\al,\beta)}\BU_r)
\left(\te\int_0^t\nabla\Bv_2(\xi,s)\intd s+(1-\te)\int_0^t\nabla\Bv_1(\xi,s)\intd s\right) \notag \\
& \cdot\pa \left(\te\int_0^t\pa_{\beta'} v_{2\al'}(\xi,s)\intd s+(1-\te)\int_0^t\pa_{\beta'} v_{1\al'}(\xi,s)\intd s\right) \intd\te \notag \\
& \cdot\left(\int_0^t\pa_\beta v_{2\al}(\xi,s)\intd s-\int_0^t\pa_\beta v_{1\al}(\xi,s)\intd s\right) \notag \\
& +
\sum_{(\al,\beta)\in\CN}\int_0^1(\pa_{(\al,\beta)}\BU_r)
\left(\te\int_0^t\nabla\Bv_2(\xi,s)\intd s+(1-\te)\int_0^t\nabla\Bv_1(\xi,s)\intd s\right)\intd\te \notag \notag \\
&\cdot\pa\left(\int_0^t\pa_\beta v_{2\al}(\xi,s)\intd s-\int_0^t\pa_\beta v_{1\al}(\xi,s)\intd s\right).
\notag
\end{align*}
By \eqref{170709_1}, \eqref{170709_3}, \eqref{170709_4}, and Lemma \ref{lemm:embed} \eqref{lemm:embed_2},
it holds that
\begin{align*}
&\left\|\pa_t\left\{\BU_r\left(\int_0^t\nabla\Bv_2(\xi,s)\intd s\right)
-\BU_r\left(\int_0^t\nabla\Bv_1(\xi,s)\intd s\right)\right\}\right\|_{L_\infty(\BR_+,L_\infty(\Om))} \\
&\leq \left(2m_2M_1^2(p'\ga_0)^{-1/p'}+m_2M_1\right)\|\Bv_2-\Bv_1\|_{Y_{q,p}^{\ga_0}}, \\
&\left\|\pa_j\left\{\BU_r\left(\int_0^t\nabla\Bv_2(\xi,s)\intd s\right)
-\BU_r\left(\int_0^t\nabla\Bv_1(\xi,s)\intd s\right)\right\}\right\|_{L_\infty(\BR_+,L_q(\Om))} \\
&\leq \left(2m_2M_1(p'\ga_0)^{-2/p'}+m_2(p'\ga_0)^{-1/p'}\right)\|\Bv_2-\Bv_1\|_{Y_{q,p}^{\ga_0}},
\end{align*}
for $j=1,\dots,N$. Similarly to the case $\BU_1$, we also have
\begin{align*}
&\left\|\BU_r\left(\int_0^t\nabla\Bv_2(\xi,s)\intd s\right)
-\BU_r\left(\int_0^t\nabla\Bv_1(\xi,s)\intd s\right)\right\|_{L_\infty(\BR_+,L_\infty(\Om))} \\
&\leq m_2M_1(p'\ga_0)^{-1/p'}\|\Bv_2-\Bv_1\|_{Y_{q,p}^{\ga_0}}, \\
&\left\|\BU_r\left(\int_0^t\nabla\Bv_2(\xi,s)\intd s\right)
-\BU_r\left(\int_0^t\nabla\Bv_1(\xi,s)\intd s\right)\right\|_{L_\infty(\BR_+,L_q(\Om))} \\
&\leq m_2(p'\ga_0)^{-1/p'}\|\Bv_2-\Bv_1\|_{Y_{q,p}^{\ga_0}}.
\end{align*}
One thus obtains the required estimates for $\BU_r$ $(r=2,\dots,7)$
by choosing a larger $M_2\geq 1$ if necessary.
This completes the proof of Lemma \ref{lemm:nonl} \eqref{lemm:nonl_1}.


\eqref{lemm:nonl_2}.
The estimates follows from \eqref{null} and \eqref{lemm:nonl_1} with $(\Bv_2, \Bv_1)=(\Bv,0)$  immediately.
This completes the proof of Lemma \ref{lemm:nonl} \eqref{lemm:nonl_2}.

\eqref{lemm:nonl_3}.
One sets $(\Bv_2, \Bv_1)=(\Bv,0)$ in the inequality proved in \eqref{lemm:nonl_1} in order to obtain
\begin{align*}
\sum_{
\stackrel{\text{\tiny{$i,j,k,l,$}}}{m,n=1}
}^N\left\|W_{ijklm}^n\left(\int_0^t\nabla\Bv(\xi,s)\intd s\right)
-W_{ijklm}^n\left(\BO\right)\right\|_{L_\infty(\BR_+,L_\infty(\Om))} 
\leq M_2\|\Bv\|_{Y_{q,p}^{\ga_0}}\leq M_2,
\end{align*}
which implies that
\begin{align*}
&\sum_{i,j,k,l,m,n=1}^N\left\|W_{ijklm}^n\left(\int_0^t\nabla\Bv(\xi,s)\intd s\right)
\right\|_{L_\infty(\BR_+,L_\infty(\Om))}  \\
&\leq 
M_2+
\sum_{i,j,k,l,m,n=1}^N\left\|W_{ijklm}^n(\BO)\right\|_{L_\infty(\BR_+,L_\infty(\Om))}.
\end{align*}
Noting $W_{ijklm}^n(\BO)\in L_\infty(\BR_+,L_\infty(\Om))$,
we complete the proof of Lemma \ref{lemm:nonl} \eqref{lemm:nonl_3}
with $M_3=M_2+\sum_{i,j,k,l,m,n=1}^N\|W_{ijklm}^n(\BO)\|_{L_\infty(\BR_+,L_\infty(\Om))}$.
\end{proof}


%
%

Let $M_1$, $M_2$, and $M_3$ be the same positive constants as in
Lemmas \ref{lemm:embed}, Lemma \ref{lemm:nonl}, 
and also assume that $(\Bv,\Fq), (\Bv_1,\Fq_1),(\Bv_2,\Fq_2)\in X_{q,p}^{\ga_0}(1)$ in what follows.

{\bf Estimates of $\BF(\Bv)$.}
One observes that
\begin{align*}
&e^{\ga_ 0 t}(\BF(\Bv_2)-\BF(\Bv_1)) \\
&=\left\{\BU_1\left(\int_0^t\nabla\Bv_2\intd s\right)-\BU_1\left(\int_0^t\nabla\Bv_1\intd s\right)\right\}e^{\ga_0 t}\pa_t\Bv_2 \\
&\quad + \BU_1\left(\int_0^t\nabla\Bv_1\intd s\right)e^{\ga_0 t}\pa_t(\Bv_2-\Bv_1) \\
&\quad+\left\{\BV\left(\int_0^t\nabla\Bv_2\intd s\right)-\BV\left(\int_0^t\nabla\Bv_1\intd s\right)\right\}e^{\ga_0 t}\nabla^2\Bv_2 \\
&\quad + \BV\left(\int_0^t\nabla\Bv_1\intd s\right)e^{\ga_0 t}\nabla^2(\Bv_2-\Bv_1) \\
&\quad+\left[\left\{\BW\left(\int_0^t\nabla\Bv_2\intd s\right)-\BW\left(\int_0^t\nabla\Bv_1\intd s\right)\right\}
\int_0^t\nabla^2\Bv_2\intd s\right]e^{\ga_0 t}\nabla\Bv_2 \\
&\quad+\left[\BW\left(\int_0^t\nabla\Bv_1\intd s\right)\int_0^t\nabla^2(\Bv_2-\Bv_1)\intd s\right]e^{\ga_0 t}\nabla\Bv_2 \\
&\quad +\left[\BW\left(\int_0^t\nabla\Bv_1\intd s\right)\int_0^t\nabla^2\Bv_1\intd s\right]e^{\ga_0 t}\nabla(\Bv_2-\Bv_1) \\
&=:I_1+\dots+I_7.
\end{align*}
By Lemma \ref{lemm:nonl}, we easily have
\begin{equation*}
\|I_j\|_{L_p(\BR_+,L_q(\Om))} 
\leq M_2\left(\|\Bv_2\|_{Y_{q,p}^{\ga_0}}+\|\Bv_1\|_{Y_{q,p}^{\ga_0}}\right)\|\Bv_2-\Bv_1\|_{Y_{q,p}^{\ga_0}}
\quad(j=1,\dots,4).
\end{equation*}
On the other hand, by Lemma \ref{lemm:embed} \eqref{lemm:embed_1}
\begin{equation*}
\left\|e^{\ga_0 t}\nabla\Bv_2\right\|_{L_p(\BR_+,L_\infty(\Om))}
\leq M_1\|e^{\ga_0 t}\nabla\Bv_2\|_{L_p(\BR_+,W_q^1(\Om))}\leq M_1\|\Bv_2\|_{Y_{q,p}^{\ga_0}}\leq M_1,
\end{equation*}
which, combined with Lemma \ref{lemm:nonl} and \eqref{170709_4}, furnishes
\begin{align*}
\|I_5\|_{L_p(\BR_+,L_q(\Om))}
&\leq 
\left\|\BW\left(\int_0^t\nabla\Bv_2\intd s\right)
-\BW\left(\int_0^t\nabla\Bv_1\intd s\right)\right\|_{L_\infty(\BR_+,L_\infty(\Om))} \\
&\quad \cdot\left\|\int_0^t\nabla^2\Bv_2\intd s\right\|_{L_\infty(\BR_+,L_q(\Om))}
\left\|e^{\ga_0 t}\nabla\Bv_2\right\|_{L_p(\BR_+,L_\infty(\Om))} \\
&\leq (p'\ga_0)^{-1/p'}M_1M_2\left(\|\Bv_2\|_{Y_{q,p}^{\ga_0}}+\|\Bv_1\|_{Y_{q,p}^{\ga_0}}\right)\|\Bv_2-\Bv_1\|_{Y_{q,p}^{\ga_0}}.
\end{align*}
Analogously, it holds that for $j=6,7$
\begin{equation*}
\|I_j\|_{L_p(\BR_+,L_q(\Om))}
\leq (p'\ga_0)^{-1/p'}M_1M_3\left(\|\Bv_2\|_{Y_{q,p}^{\ga_0}}+\|\Bv_1\|_{Y_{q,p}^{\ga_0}}\right)\|\Bv_2-\Bv_1\|_{Y_{q,p}^{\ga_0}}.
\end{equation*}
Summing up the above estimates for $I_1,\dots,I_7$, we have achieved 
\begin{align}\label{170527_12}
&\|e^{\ga_ 0 t}(\BF(\Bv_2)-\BF(\Bv_1))\|_{L_p(\BR_+,L_q(\Om))} \\
&\leq \left(4M_2+(p'\ga_0)^{-1/p'}M_1M_2+2(p'\ga_0)^{-1/p'}M_1M_3\right) \notag \\
&\quad \cdot \left(\|\Bv_2\|_{Y_{q,p}^{\ga_0}}+\|\Bv_1\|_{Y_{q,p}^{\ga_0}}\right)
\|\Bv_2-\Bv_1\|_{Y_{q,p}^{\ga_0}}.
\notag
\end{align}
Especially, setting $(\Bv_2,\Bv_1)=(\Bv,0)$ in \eqref{170527_12} yields 
\begin{align}\label{170527_11}
&\|e^{\ga_0 t}\BF(\Bv)\|_{L_p(\BR_+,L_q(\Om))} \\
&\leq \left(4M_2+(p'\ga_0)^{-1/p'}M_1M_2+2(p'\ga_0)^{-1/p'}M_1M_3\right)\|\Bv\|_{Y_{q,p}^{\ga_0}}^2.
\notag
\end{align}

{\bf Estimates of $G(\Bv)$, $\BG(\Bv)$.}
One observes that
\begin{align*}
&e^{\ga_0 t}(G(\Bv_2)-G(\Bv_1)) 
=\left\{\BU_2\left(\int_0^t\nabla\Bv_2 \intd s\right)-\BU_2\left(\int_0^t\nabla\Bv_1 \intd s\right)\right\}:e^{\ga_0 t}\nabla\Bv_2 \\
&\quad + \BU_2\left(\int_0^t\nabla\Bv_1 \intd s\right):e^{\ga_0 t}(\nabla\Bv_2-\nabla\Bv_1) 
=:I_8+I_9.
\end{align*}
By Lemmas \ref{lemm:H1/2}, \ref{lemm:nonl}, we have
\begin{align*}
&\|I_8\|_{H_{q,p}^{1,1/2}(\Om\times\BR_+)} \\
&\leq \left\|\BU_2\left(\int_0^t\nabla\Bv_2 \intd s\right)-\BU_2\left(\int_0^t\nabla\Bv_1 \intd s\right)
\right\|_{B_q(\Om\times\BR_+)}\|e^{\ga_0 t}\nabla\Bv_2\|_{H_{q,p}^{1,1/2}(\Om\times\BR_+)} \\
&\leq M_2\|\Bv_2-\Bv_1\|_{Y_{q,p}^{\ga_0}}\left(\|e^{\ga_0 t}\Bv_2\|_{H_p^{1/2}(\BR_+,W_q^1(\Om))}+\|e^{\ga_0 t}\Bv_2\|_{L_p(\BR_+,W_q^2(\Om))}\right).
\end{align*}
Since it holds by Lemma \ref{lemm:embed} \eqref{lemm:embed_3} that 
\begin{align*}
\|e^{\ga_0 t}\Bv_2\|_{H_p^{1/2}(\BR_+,W_q^1(\Om))}
&\leq M_1\|e^{\ga_0 t}\Bv_2\|_{W_{q,p}^{2,1}(\Om\times\BR_+)} \\
&\leq M_1(\ga_0+1)\|\Bv_2\|_{Y_{q,p}^{\ga_0}},
\end{align*}
the above inequality for $I_8$ yields
\begin{equation*}
\|I_8\|_{H_{q,p}^{1,1/2}(\Om\times\BR_+)}
\leq 2M_1M_2(\ga_0+1)\|\Bv_2\|_{Y_{q,p}^{\ga_0}}\|\Bv_2-\Bv_1\|_{Y_{q,p}^{\ga_0}}.
\end{equation*}
Analogously, we have
\begin{equation*}
\|I_9\|_{H_{q,p}^{1,1/2}(\Om\times\BR_+)}
\leq 2M_1M_2(\ga_0+1)\|\Bv_1\|_{Y_{q,p}^{\ga_0}}\|\Bv_2-\Bv_1\|_{Y_{q,p}^{\ga_0}},
\end{equation*}
and therefore
\begin{align}\label{170630_1}
&\|e^{\ga_0 t}(G(\Bv_2)-G(\Bv_1))\|_{H_{q,p}^{1,1/2}(\Om\times\BR_+)} \\
&\leq 2M_1M_2(\ga_0+1)\left(\|\Bv_2\|_{Y_{q,p}^{\ga_0}}+\|\Bv_1\|_{Y_{q,p}^{\ga_0}}\right)
\|\Bv_2-\Bv_1\|_{Y_{q,p}^{\ga_0}}.
\notag
\end{align}
In addition, it is clear that $e^{\ga_0 t}(G(\Bv_2)-G(\Bv_1))|_{t=0}=0$ in $L_q(\Om)$,
which, combined with \eqref{chara:1} and \eqref{170630_1}, furnishes 
\begin{equation}\label{170630_2}
e^{\ga_0 t}(G(\Bv_2)-G(\Bv_1))\in {}_0H_{q,p}^{1,1/2}(\Om\times\BR_+).
\end{equation}
Especially, setting $(\Bv_2,\Bv_1)=(\Bv,0)$ in \eqref{170630_2} and \eqref{170630_1}
yields, respectively,
\begin{align}\label{170630_3}
&e^{\ga_0 t}G(\Bv)\in {}_0H_{q,p}^{1,1/2}(\Om\times\BR_+), \\
&\|e^{\ga_0 t}G(\Bv)\|_{H_{q,p}^{1,1/2}(\Om\times\BR_+)}
\leq 2M_1M_2(\ga_0+1)\|\Bv\|_{Y_{q,p}^{\ga_0}}^2.
\notag
\end{align}

Next we show $e^{\ga_0 t}G(\Bv)\in {}_0W_p^1(\BR_+,\wh {W}_{q,\Ga}^{-1}(\Om))$ and its estimate.
For any $\ph\in W_{q',\Ga}^1(\Om)$ with $q'=q/(q-1)$, we have for $k=0,1$
\begin{align*}
(\pa_t^k(G(\Bv_2)-G(\Bv_1)),\ph)_{\Om}&=(\pa_t^k\di(\BG(\Bv_2)-\BG(\Bv_1)),\ph)_{\Om} \\
&=-(\pa_t^k(\BG(\Bv_2)-\BG(\Bv_1)),\nabla\ph)_{\Om},
\end{align*}
where we have used $\BG(\Bv_l)=0$ on $S$ for $l=1,2$.
Thus $\pa_t^k(\BG(\Bv_2)-\BG(\Bv_1))\in\CG(\pa_t^k(G(\Bv_2)-G(\Bv_1)))$ as was discussed in Subsection \ref{subsec:main},
and also
\begin{equation*}
\|\pa_t^k(G(\Bv_2)-G(\Bv_1))\|_{\wh W_{q,\Ga}^{-1}(\Om)}\leq \|\pa_t^k(\BG(\Bv_2)-\BG(\Bv_1))\|_{L_q(\Om)}.
\end{equation*}
This inequality, together with $e^{\ga_0 t}(\BG(\Bv_2)-\BG(\Bv_1))|_{t=0}=0$ in $L_q(\Om)^N$,
implies that
\begin{align}\label{170722_1}
&e^{\ga_0 t}(G(\Bv_2)-G(\Bv_1))|_{t=0}=0 \quad \text{in $\wh W_{q,\Ga}^{-1}(\Om)$,} \\
&\|e^{\ga_0 t}\pa_t^k(G(\Bv_2)-G(\Bv_1))\|_{L_p(\BR_+,\wh W_{q,\Ga}^{-1}(\Om))} \notag \\
& \leq \|e^{\ga_0 t}\pa_t^k(\BG(\Bv_2)-\BG(\Bv_1))\|_{L_p(\BR_+,L_q(\Om))} \quad (k=0,1). \notag
\end{align}

From now on, we estimate the right-hand side of the inequality in \eqref{170722_1}.
To this end, we set
\begin{align*}
&e^{\ga_0 t}(\BG(\Bv_2)-\BG(\Bv_1)) 
=\left(\BU_3\left(\int_0^t\nabla\Bv_2 \intd s\right)-\BU_3\left(\int_0^t\nabla\Bv_1 \intd s\right)\right)e^{\ga_0 t}\Bv_2 \\
& \quad +\BU_3\left(\int_0^t\nabla\Bv_1 \intd s\right)e^{\ga_0 t}(\Bv_2-\Bv_1) =: I_{10}+I_{11}, \\
&e^{\ga_0 t}\pa_t(\BG(\Bv_2)-\BG(\Bv_1)) 
=\left\{\pa_t\left(\BU_3\left(\int_0^t\nabla\Bv_2 \intd s\right)-\BU_3\left(\int_0^t\nabla\Bv_1 \intd s\right)\right)\right\}e^{\ga_0 t}\Bv_2 \\
&\quad + \left(\BU_3\left(\int_0^t\nabla\Bv_2 \intd s\right)-\BU_3\left(\int_0^t\nabla\Bv_1 \intd s\right)\right)e^{\ga_0 t}\pa_t\Bv_2 \\
& \quad +\left\{\pa_t\BU_3\left(\int_0^t\nabla\Bv_1 \intd s\right)\right\}e^{\ga_0 t}(\Bv_2-\Bv_1) \\
& \quad +\BU_3\left(\int_0^t\nabla\Bv_1 \intd s\right)e^{\ga_0 t}\pa_t(\Bv_2-\Bv_1)
=: I_{12}+I_{13}+I_{14}+I_{15}.
\end{align*}
By Lemma \ref{lemm:nonl}, we see that for $j=10,\dots,15$
\begin{equation*}
\|I_j\|_{L_p(\BR_+,L_q(\Om))}
\leq M_2\left(\|\Bv_2\|_{Y_{q,p}^{\ga_0}}+\|\Bv_1\|_{Y_{q,p}^{\ga_0}}\right)
\|\Bv_2-\Bv_1\|_{Y_{q,p}^{\ga_0}}.
\end{equation*}
It thus holds that
\begin{align*}
&\sum_{k=0}^1\|e^{\ga_0 t}\pa_t^k(\BG(\Bv_2)-\BG(\Bv_1))\|_{L_p(\BR_+,L_q(\Om))} \\
&\leq 6M_2\left(\|\Bv_2\|_{Y_{q,p}^{\ga_0}}+\|\Bv_1\|_{Y_{q,p}^{\ga_0}}\right)
\|\Bv_2-\Bv_1\|_{Y_{q,p}^{\ga_0}},
\end{align*}
which, combined with \eqref{170722_1}, furnishes
\begin{align}\label{170722_5}
&e^{\ga_0 t}(G(\Bv_2)-G(\Bv_1))\in {}_0 W_p^1(\BR_+,\wh W_{q,\Ga}^{-1}(\Om)), \\
&\|e^{\ga_0 t}(G(\Bv_2)-G(\Bv_1))\|_{W_p^1(\BR_+,\wh W_{q,\Ga}^{-1}(\Om))} \notag \\
&\leq 6M_2(\ga_0+1)\left(\|\Bv_2\|_{Y_{q,p}^{\ga_0}}+\|\Bv_1\|_{Y_{q,p}^{\ga_0}}\right)
\|\Bv_2-\Bv_1\|_{Y_{q,p}^{\ga_0}}.
\notag
\end{align}
%
%
%
%
Especially, setting $(\Bv_2,\Bv_1)=(\Bv,0)$ in \eqref{170722_5} 
yields 
\begin{align}\label{170630_12}
&e^{\ga_0 t}G(\Bv)\in {}_0W_p^1(\BR_+,\wh W_{q,\Ga}^{-1}(\Om)), \\
&\|e^{\ga_0 t}G(\Bv)\|_{W_p^1(\BR_+,\wh W_{q,\Ga}^{-1}(\Om))}
\leq 6M_2(\ga_0+1)\|\Bv\|_{Y_{q,p}^{\ga_0}}^2. \notag 
\end{align}

{\bf Etimates of $\BH(\Bv)$.}
In the same manner as in the case $G(\Bv)$, we have
\begin{align} \label{170528_7}
&e^{\ga_0 t}(\BH(\Bv_2)\Be_N-\BH(\Bv_1)\Be_N),e^{\ga_0 t}\BH(\Bv)\Be_N\in {}_0 H_{q,p}^{1,1/2}(\Om\times\BR_+)^N, \\
&\left\|e^{\ga_0 t}(\BH(\Bv_2)\Be_N-\BH(\Bv_1)\Be_N)\right\|_{H_{q,p}^{1,1/2}(\Om\times\BR_+)} \notag \\
&\leq \left(8M_1M_2(\ga_0+1)+6M_1M_2^2(\ga_0+1)\right) \notag \\
&\quad \cdot\left(\|\Bv_2\|_{Y_{q,p}^{\ga_0}}+\|\Bv_2\|_{Y_{q,p}^{\ga_0}}\right)
\|\Bv_2-\Bv_1\|_{Y_{q,p}^{\ga_0}}, \notag \\
&\left\|e^{\ga_0 t}\BH(\Bv)\Be_N\right\|_{H_{q,p}^{1,1/2}(\Om\times\BR_+)} \notag \\
&\leq \left(8M_1M_2(\ga_0+1)+6M_1M_2^2(\ga_0+1)\right)\|\Bv\|_{Y_{q,p}^{\ga_0}}^2.
\notag
\end{align}

Let $M_4\geq 1$ be a positive constant defined as
\begin{align}\label{170711_1}
M_4&=4M_2+(p'\ga_0)^{-1/p'}M_1M_2+2(p'\ga_0)^{-1/p'}M_1M_3+2M_1M_2(\ga_0+1) \\
&+6M_2(\ga_0+1)+8M_1M_2(\ga_0+1)+6M_1M_2^2(\ga_0+1).
\notag
\end{align}
Summing up \eqref{170527_12}-\eqref{170630_3} and \eqref{170722_5}-\eqref{170528_7},
we have obtained
\begin{align}\label{170722_7}
&e^{\ga_0 t}(\BF(\Bv_2)-\BF(\Bv_1)), e^{\ga_0 t}\BF(\Bv)
\in L_p(\BR_+,L_q(\Om))^N, \\
&e^{\ga_0 t}(G(\Bv_2)-G(\Bv_1)),  e^{\ga_0 t}G(\Bv)
\in {}_0H_{q,p}^{1,1/2}(\Om\times\BR_+) \cap {}_0W_p^1(\BR_+,\wh W_{q,\Ga}^{-1}(\Om)), \notag \\
&e^{\ga_0 t}(\BH(\Bv_2)\Be_N-\BH(\Bv_1)\Be_N), e^{\ga_0 t}\BH(\Bv)\Be_N
\in {}_0H_{q,p}^{1,1/2}(\Om\times\BR_+)^N,
\notag
\end{align}
together with the estimates:
\begin{align}
&\left\|e^{\ga_0 t}(\BF(\Bv_2)-\BF(\Bv_1))\right\|_{L_p(\BR_+,L_q(\Om))} \label{170528_10} \\
&\quad+\left\|e^{\ga_0 t}(\BH(\Bv_2)\Be_N-\BH(\Bv_1)\Be_N)\right\|_{H_{q,p}^{1,1/2}(\Om\times\BR_+)} \notag  \\
&\quad +\left\|e^{\ga_0 t}(G(\Bv_2)-G(\Bv_1))\right\|_{H_{q,p}^{1,1/2}(\Om\times\BR_+)} \notag  \\
&\quad +\left\|e^{\ga_0 t}(G(\Bv_2)-G(\Bv_1))\right\|_{W_p^1(\BR_+,\wh W_{q,\Ga}^{-1}(\Om))} \notag  \\
&\leq 4M_4\left(\|\Bv_2\|_{Y_{q,p}^{\ga_0}}+\|\Bv_1\|_{Y_{q,p}^{\ga_0}}\right)\|\Bv_2-\Bv_1\|_{Y_{q,p}^{\ga_0}}. \notag \\
&\left\|e^{\ga_0 t}\BF(\Bv)\right\|_{L_p(\BR_+,L_q(\Om))}
+\left\|e^{\ga_0 t}\BH(\Bv)\Be_N\right\|_{H_{q,p}^{1,1/2}(\Om\times\BR_+)} \label{170528_9} \\
&\quad+\left\|e^{\ga_0 t}G(\Bv)\right\|_{H_{q,p}^{1,1/2}(\Om\times\BR_+)}
+ \left\|e^{\ga_0 t}G(\Bv)\right\|_{W_p^1(\BR_+,\wh W_{q,\Ga}^{-1}(\Om)}
\leq 4M_4\|\Bv\|_{Y_{q,p}^{\ga_0}}^2. \notag 
\end{align}

We now  construct a solution to \eqref{170522_1} by using Thorem \ref{theo:main2}.
For any $(\Bv,\Fq)\in X_{q,p}^{\ga_0}(\de_0)$ with $0<\de_0<1$,
we have by \eqref{170722_7} a unique solution 
$(\Bu,\Fp)\in W_{q,p}^{2,1}(\Om\times\BR_+)^N\times L_p(\BR_+,W_q^1(\Om))$
to \eqref{170522_1} 
and have by \eqref{170528_9}
\begin{equation*}
\|(\Bu,\Fp)\|_{X_{q,p}^{\ga_0}}
\leq c_0(\|\Ba\|_{B_{q,p}^{2-2/p}(\Om)}+4M_4\|\Bv\|_{Y_{q,p}^{\ga_0}}^2)
\leq c_0(\ep_0+4M_4\de_0^2).
\end{equation*}
One here chooses $\de_0$, $\ep_0$ as follows:
\begin{equation}\label{delep_0}
4c_0 M_4\de_0^2 \leq \frac{\de_0}{4}, \quad c_0\ep_0 \leq \frac{\de_0}{2},
\end{equation}
which enable us to define the operator:
\begin{equation*}
\Phi: X_{q,p}^{\ga_0}(\de_0)\ni(\Bv,\Fq)\mapsto \Phi(\Bv,\Fq)=(\Bu,\Fp)\in X_{q,p}^{\ga_0}(\de_0).
\end{equation*}

Let $(\Bv_i,\Fq_i)\in X_{q,p}^{\ga_0}(\de_0)$ 
and $(\Bu_i,\Fp_i)=\Phi(\Bv_i,\Fq_i)$ for $i=1,2$.
Setting $\bar\Bu=\Bu_2-\Bu_1$ and $\bar\Fp=\Fp_2-\Fp_1$, we observe that
\begin{equation*}
\left\{\begin{aligned}
\pa_t\bar\Bu-\Di\BT(\bar\Bu,\bar\Fp) = \BF(\Bv_2)-\BF(\Bv_1)& && \text{in $\Om$, $t>0$,} \\
\di\bar\Bu = G(\Bv_2)-G(\Bv_1) = \di(\BG(\Bv_2)-\BG(\Bv_1))& && \text{in $\Om$, $t>0$,} \\
\BT(\bar\Bu,\bar\Fp)\Be_N = (\BH(\Bv_2)-\BH(\Bv_1))\Be_N& && \text{on $\Ga$, $t>0$,} \\
\bar\Bu = 0& && \text{on $S$, $t>0$,} \\
\bar\Bu|_{t=0} = 0& && \text{in $\Om$.}
\end{aligned}\right.
\end{equation*}
Together with \eqref{170722_7}, \eqref{170528_10}, and \eqref{delep_0},
one has by Thorem \ref{theo:main2}
\begin{equation*}
\|(\bar\Bu,\bar\Fp)\|_{X_{q,p}^{\ga_0}}
\leq 4c_0M_4(\de_0+\de_0)\|\Bv_2-\Bv_1\|_{Y_{q,p}^{\ga_0}}
\leq\frac{1}{2}\|(\Bv_2,\Fq_2)-(\Bv_1,\Fq_1)\|_{X_{q,p}^{\ga_0}},
\end{equation*}
which implies that $\Phi$ is a contraction mapping on $X_{q,p}^{\ga_0}(\de_0)$.
The contraction mapping theorem thus proves that 
there exists a unique fixed point $(\Bu_*,\Fp_*)$ of $\Phi$ in $X_{q,p}^{\ga_0}$.
Such a $(\Bu_*,\Fp_*)$ is a unique solution to \eqref{eq:11}-\eqref{eq:15} in $X_{q,p}^{\ga_0}(\de_0)$.
This completes the proof of Theorem \ref{theo:main}.

\section{Original nonlinear problem}\label{sec:original}
This section is concerned with the global solvability of
the original nonlinear problem \eqref{eq:2}-\eqref{eq:7}.
We first introduce the definition of $L_p\text{-}L_q$ solutions to \eqref{eq:2}-\eqref{eq:7}.
Next we show the global existence and uniqueness of such solutions,
and also their exponential stability.
Let $M_1$, $M_2$, $M_3$, and $M_4$ be the same positive constants as in Section \ref{sec:nonl},
and let $c_0$ be the positive constant given by Theorem \ref{theo:main2}.

\subsection{Definition of $L_p\text{-}L_q$ solutions}
Following \cite{EPS03},
we introduce the definition of $L_p\text{-}L_q$ solutions for \eqref{eq:2}-\eqref{eq:7} in this subsection.

One first recalls the definition of $L_p\text{-}L_q$ solutions for the equations \eqref{eq:11}-\eqref{eq:15}.

\begin{defi}\label{defi:fix}
We call a pair $(\Bu,\Fp)$ an $L_p\text{-}L_q$ solution global in time to \eqref{eq:11}-\eqref{eq:15} if
$(\Bu,\Fp) \in W_{q,p}^{2,1}(\Om\times\BR_+)^N\times L_p(\BR_+ W_q^1(\Om))$
and if $(\Bu,\Fp)$ satisfies \eqref{eq:11}-\eqref{eq:15} in the $L_p\text{-}L_q$ sense 
for some $1<p,q<\infty$ and $\Ba\in B_{q,p}^{2-2/p}(\Om)^N$.
\end{defi}

\begin{rema}
Due to  \cite{Shibata15},
the maximal $L_p\text{-}L_q$ regularity class means
the function space of $(\Bu,\Fp)$ in Definition \ref{defi:fix}.
\end{rema}

Then we can define $L_p\text{-}L_q$ solutions to \eqref{eq:2}-\eqref{eq:7} as follows:

\begin{defi}\label{defi:free}
We call a triplet $(\Te,\Bv,\pi)$, where for $\Om(t)=\Te(\Om,t)$
\begin{equation*}
\Bv:\bigcup_{t\in(0,\infty)}\left(\Om(t)\times \{t\}\right)\to\BR^N, \quad
\pi:\bigcup_{t\in(0,\infty)}\left(\Om(t)\times \{t\}\right)\to\BR,
\end{equation*}
an $L_p\text{-}L_q$ solution global in time to \eqref{eq:2}-\eqref{eq:7}
if the following assertions hold true for some $1<p,q<\infty$ and $\Ba\in B_{q,p}^{2-2/p}(\Om)^N:$
\begin{enumerate}[$(1)$]
\item\label{free:1}
$\Te=\Te(\xi,t)$ is a solution to \eqref{eq:2} in the classical sense.
\item\label{free:2}
$\Te(\cdot,t)$ is a $C^1$-diffeomorphism from $\Om$ onto $\Om(t)$ for each $t>0$.
\item\label{free:3}
$(\Bu,\Fp)=(\Bv\circ\Te,\pi\circ\Te)$ is an $L_p\text{-}L_q$ solution global in time 
to \eqref{eq:11}-\eqref{eq:15}.
\end{enumerate}
\end{defi}



\subsection{Global solvability and exponential stability}

We here prove

\begin{theo}\label{theo:ex}
Let $p$, $q$ satisfy \eqref{pq}.
Suppose that $\ep_0$ is the same positive number as in Theorem $\ref{theo:main}$
and that $\Ba\in D_{q,p}(\Om)$ with $\|\Ba\|_{D_{q,p}(\Om)}\leq\ep_0$.
Then there exists an $L_p\text{-}L_q$ solution $(\Te,\Bv,\pi)$ global in time to \eqref{eq:2}-\eqref{eq:7}, which is unique.
In addition, 
$\|\Bv(t)\|_{L_q(\Om(t))}=O(e^{-\ga_0 t})$ as $t\to\infty$,
where $\ga_0$ is the same positive constant as in Theorem $\ref{theo:main}$. 
\end{theo}

\begin{proof}
By Theorem $\ref{theo:main}$, we have an $L_p\text{-}L_q$ solution $(\Bu,\Fp)$ to \eqref{eq:11}-\eqref{eq:15}. 
Let us define for $t>0$
\begin{equation*}
\Te(\xi,t)=\xi+\int_0^t\Bu(\xi,s)\intd s \quad (\xi\in\Om), \quad \Om(t)=\Te(\Om,t),
\end{equation*}
where we note that $\Bu\in BUC([0,\infty),BUC^1(\overline{\Om}))$ by Lemma \ref{lemm:embed} \eqref{lemm:embed_2}.

{\bf Step 1.} 
In this step,
we prove the existence of $L_p\text{-}L_q$ solutions global in time to the equations \eqref{eq:2}-\eqref{eq:7}.

Let $x_1,x_2\in\Om(t)$ with $x_1=x_2$ for
\begin{equation*}
x_i=\xi_i+\int_0^t\Bu(\xi_i,s)\intd s \quad (i=1,2),   
\end{equation*}
where $\xi_1,\xi_2\in\Om$. Since it holds that by Lemma \ref{lemm:embed} \eqref{lemm:embed_1}, \eqref{170711_1}, and \eqref{delep_0} 
\begin{align}\label{170723_1}
&\int_0^t\|\nabla\Bu(\cdot,s)\|_{BUC(\overline{\Om})}\intd s
\leq\left(\int_0^\infty e^{-p'\ga_0s}\intd s\right)^{1/p'}\|e^{\ga_0 t}\nabla \Bu\|_{L_p(\BR_+,BUC(\overline{\Om}))} \\
&\leq (p'\ga_0)^{-1/p'}M_1\|e^{\ga_0 t}\nabla\Bu\|_{L_p(\BR_+,W_q^1(\Om))} \notag \\
&\leq (p'\ga_0)^{-1/p'}M_1\de_0 \leq (p'\ga_0)^{-1/p'}M_1M_2\de_0 \leq 4c_0M_4\de_0 \leq \frac{1}{4} \notag
\end{align}
for any $t>0$, one observes that 
\begin{align*}
0&=|x_1-x_2|\geq |\xi_1-\xi_2|-\int_0^t|\Bu(\xi_1,s)-\Bu(\xi_2,s)|\intd s \\
&\geq |\xi_1-\xi_2|-|\xi_1-\xi_2|\int_0^t\|\nabla\Bu(\cdot,s)\|_{BUC(\overline{\Om})}\intd s \\
&\geq \frac{3}{4}|\xi_1-\xi_2|.
\end{align*}
This inequality implies $\xi_1=\xi_2$, and thus
$\Te(\cdot,t)$ is bijective from $\Om$ onto $\Om(t)$ for each $t>0$.
We denote inverse mapping of $\Te(\cdot,t)$  by $\Te^{-1}(\cdot,t)$ in what follows.
Here $\Te(\cdot,t):\Om\to\Om(t)$ is a $C^1$ function,
so that we see by the inverse function theorem that $\Te^{-1}(\cdot,t):\Om(t)\to\Om$ is also a $C^1$ function.  
Hence, $\Te(\cdot,t)$ satisfies Definition \ref{defi:free} \eqref{free:2},
while it is clear that
\begin{equation}\label{170710_13}
\Bv(x,t)=\Bu(\Te^{-1}(x,t),t), \quad \pi(x,t)=\Fp(\Te^{-1}(x,t),t) \quad 
(x\in\Om(t), t>0)
\end{equation}
satisfies Definition \ref{defi:free} \eqref{free:3} and $\Te$ is a solution to \eqref{eq:2}.

%
%

{\bf Step 2.}
Let $\BJ_\Te$ be the Jacobian matrix of $\Te$. 
Since $\BJ_\Te=\BI+\int_0^t\nabla\Bu(\xi,s)\intd s$, one has 
$|\BJ_\Te|\leq c_1$ by \eqref{170723_1}
for some positive constant $c_1$.
It then holds by \eqref{170710_13} that
\begin{align*}
\|\Bv(t)\|_{L_q(\Om(t))}
&= \left(\int_\Om|\Bv(\Te(\xi,t),t)|^q|\BJ_\Te|\intd\xi\right)^{1/q}  \\
&=\left(\int_\Om|\Bu(\xi,t)|^q|\BJ_\Te|\intd \xi\right)^{1/q} 
\leq (c_1)^{1/q}\|\Bu(t)\|_{L_q(\Om)} \\
&\leq (c_1)^{1/q}e^{-\ga_0 t}\|e^{\ga_0 t}\Bu\|_{BUC([0,\infty),L_q(\Om))} \\
&\leq (c_1)^{1/q}e^{-\ga_0 t}\|e^{\ga_0 t}\Bu\|_{BUC([0,\infty),B_{q,p}^{2-2/p}(\Om))}.
\end{align*}
This inequality, together with \eqref{est:main} and \eqref{170615_1},
furnishes the exponential stability of the solution $\Bv$.


{\bf Step 3.} We prove the uniqueness of solutions in this step,
Let $i=1,2$ and
$(\Te_i,\Bv_i,\pi_i)$ be solutions to \eqref{eq:2}-\eqref{eq:7} with $\Ba\in D_{q,p}(\Om)$ satisfying $\|\Ba\|_{D_{q,p}(\Om)}\leq \ep_0$.
Then $(\Bu_i,\Fp_i)=(\Bv_i\circ\Te_i,\pi_i\circ\Te_i)$
are $L_p\text{-}L_q$ solutions global in time to \eqref{eq:11}-\eqref{eq:15}.
By Theorem \ref{theo:main}, we see that $\Bu_1=\Bu_2$ and $\Fp_1=\Fp_2$.
One integrates \eqref{eq:2} with respect to time $t$ in order to obtain
\begin{align*}
\Te_1(\xi,t)&=\xi+\int_0^t\Bv_1(\Te_1(\xi,s),s)\intd s 
= \xi+\int_0^t\Bu_1(\xi,s)\intd s \\
&=\xi+\int_0^t\Bu_2(\xi,s)\intd s=\xi+\int_0^t\Bv_2(\Te_2(\xi,s),s)\intd s=\Te_2(\xi,t).
\end{align*}
Furthermore, for $t>0$ and $x\in \Te_1(\Om,t)=\Te_2(\Om,t)$, we observe that
\begin{equation*}
\Bv_1(x,t)=\Bv_1(\Te_1(\xi,t),t)=\Bu_1(\xi,t)=\Bu_2(\xi,t)=\Bv_2(\Te_2(\xi,t),t)=\Bv_2(x,t)
\end{equation*}
and that $\pi_1(x,t)=\pi_2(x,t)$.
This completes the proof of Theorem \ref{theo:ex}.
\end{proof}

\begin{rema}
One can prove similarly further properties of the solution $(\Te,\Bv,\pi)$ as follows:
\begin{enumerate}[(1)]
\item
Let $\Ga(t)=\Te(\Ga,t)$ for $t>0$.
Then $\Te(\cdot,t)$ is a $C^1$-diffeomorphism from $\Ga$ onto $\Ga(t)$ for each $t>0$. 
In addition, it holds that $\Te(\overline{\Om},t)=\overline{\Om(t)}$. 
\item
Let $\wt\Te(\xi,t)=(\Te(\xi,t),t)$ for $(\xi,t)\in\Om\times\BR_+$.
Then $\wt\Te$ is a $C^1$-diffeomorphism from $\Om\times\BR_+$ onto $\bigcup_{t\in(0,\infty)}(\Om(t)\times \{t\})$.
\item
$\|\nabla\Bv(t)\|_{L_q(\Om(t))}=O(e^{-\ga_0 t})$ as $t\to \infty$.
\end{enumerate}
\end{rema}

\medskip
\noindent{\bf Acknowledgments.}
This research was partly supported by JSPS Grant-in-aid for Young Scientists (B) \#17K14224,
JSPS Japanese-German Graduate Externship at Waseda University,
and Waseda University Grant for Special Research Projects (Project number: 2017K-176).

\def\thesection{A}
\renewcommand{\theequation}{A.\arabic{equation}}
\section{}
Following \cite[Appendix]{SS07},
we derive the equations \eqref{eq:11}-\eqref{eq:13} in this appendix.
To this end, we assume that the equations \eqref{eq:2}-\eqref{eq:7} have
a sufficiently regular solution $(\Te,\Bv,\pi)$ in the following argumentation.
Let $x=\Te(\xi,t)$ with \eqref{trans}, and recall that  
$\Bu(\xi,t)=\Bv(\Te(\xi,t),t)=\Bv(x,t)$ and $\Fp(\xi,t)=\pi(\Te(\xi,t),t)=\pi(x,t)$.
If we write $\BM=(M_{i j})$, then $\BM$ is an $N \times N$ matrix
whose $(i,j)$-component is $M_{i j}$.
In addition, $\de_{i j}$ denotes Kronecker's delta defined by the formula:
$\de_{i j}=1$ when $i=j$ and $\de_{i j}=0$ when $i \neq j$.

{\bf Case \eqref{eq:11}.}
It is clear that
\begin{equation}\label{150303_1}
\frac{\pa}{\pa t}\Bu(\xi,t)
= \frac{\pa}{\pa t}\Bv(x,t) + \sum_{j=1}^N\frac{\pa x_j}{\pa t}\frac{\pa}{\pa x_j}\Bv(x,t)
= \frac{\pa}{\pa t}\Bv + (\Bv\cdot\nabla)\Bv.
\end{equation}

We here set
\begin{equation*}
\BA = \left(\frac{\pa x_i}{\pa \xi_j}\right) = \BI + \BB, \quad
\BB = (B_{i j}), \quad
B_{i j} = \int_0^t \frac{\pa u_i}{\pa \xi_j}(\xi,s)\intd s.
\end{equation*}
Let $\CA$ be the cofactor matrix of $\BA$, i.e. $\BA^{-1}=(\det\BA)^{-1}\CA$.
Then $\CA=\BA^{-1}$ because $\det\BA=1$ (cf. e.g. \cite[page 271]{SS07}),
and also one observes by $\de_{ij}=\pa\xi_i/\pa\xi_j=\sum_{k=1}^N \frac{\pa\xi_i}{\pa x_k}\frac{\pa x_k}{\pa\xi_j}$ that
\begin{equation*}
\CA = \BA^{-1}= 
\begin{pmatrix}
\frac{\pa\xi_1}{\pa x_1} & \dots & \frac{\pa \xi_1}{\pa x_N} \\
\vdots & \ddots & \vdots \\
\frac{\pa\xi_N}{\pa x_1} & \dots & \frac{\pa \xi_N}{\pa x_N}
\end{pmatrix}.
\end{equation*}
In addition,
\begin{equation*}
\CA = \BI + \CB \quad
\text{for some $N\times N$ matrix $\CB=(\CB_{ij})=\CB\left(\int_0^t\nabla\Bu(\xi,s)\intd s\right)$,}
\end{equation*}
where $\CB:\BR^{N\times N}\to\BR^{N\times N}$ with $\CB(\BO)=\BO$
and the $(i,j)$-component $\CB_{i j}(\BX)$ of $\CB(\BX)$, $\BX=(X_{kl})$,
are polynomials with respect to $X_{kl}$ for $k,l=1,\dots,N$.

Now it holds, by direct calculations and by \cite[page 771]{DS95}, that
\begin{align}\label{150303_2}
&\nabla_x = \CA^\SST\nabla_\xi, \quad
\di_x = (\CA^\SST:\nabla_\xi\,\cdot\,)=\di_\xi(\CA\,\cdot\,), \\
&\nabla_x\di_x = \CA^\SST \nabla_\xi \di_\xi+\CA^\SST\nabla_\xi(\CB^\SST:\nabla_\xi\,\cdot\,), \notag
\end{align}
where the subscripts $x$ and $\xi$ denote their variables, and also
\begin{align}\label{150304_10}
&\De_x = \di_x\nabla_x = (\di_\xi + \CB^\SST:\nabla_\xi\,\cdot\,)(\BI+\CB^\SST)\nabla_\xi  \\
&= \De_\xi +\di_\xi\CB^\SST\nabla_\xi + \CB^\SST:\nabla_\xi\nabla_\xi
+ \CB^\SST:\nabla_\xi\CB^\SST\nabla_\xi \notag \\
&= \De_\xi + \sum_{i,j,k=1}^N\CB_{ji}\left(2\de_{ik}+\CB_{ki}\right)\frac{\pa^2}{\pa\xi_j \pa\xi_k} 
+\sum_{i,j,k=1}^N\left(\de_{ij}+\CB_{ji}\right)\left(\frac{\pa \CB_{ki}}{\pa \xi_j}\right)\frac{\pa}{\pa\xi_k}. \notag
\end{align}
Then we can write
\begin{align}\label{150303_3}
&\sum_{i,j,k=1}^N\CB_{ji}\left(2\de_{ik}+\CB_{ki}\right)\frac{\pa^2}{\pa\xi_j \pa\xi_k}
= \sum_{i,j,k=1}^N \CV_{ijk}\left(\int_0^t\nabla\Bu(\xi,s)\intd s\right)\frac{\pa^2}{\pa\xi_j \pa\xi_k}, \\
&\sum_{i,j,k=1}^N\left(\de_{ij}+\CB_{ji}\right)\left(\frac{\pa\CB_{ki}}{\pa\xi_j}\right)\frac{\pa}{\pa\xi_k} \notag \\
&=\sum_{i,j,k=1}^N \left[\CW_{ijk}\left(\int_0^t\nabla\Bu(\xi,s)\intd s\right)
\int_0^t\nabla^2\Bu(\xi,s)\intd s\right] \frac{\pa}{\pa\xi_k} \notag
\end{align}
by $\CV_{ijk}(\cdot):\BR \to\BR $ with $\CV_{ijk}(\BO)=0$
and by $\CW_{ijk}(\cdot):\BR^{N^3}\to\BR$. 
Note that both $\CV_{ijk}(\BX)$ and $\CW_{ijk}(\BX)$, $\BX=(X_{lm})$,
are polynomials with respect to $X_{lm}$ for $l,m=1,\dots,N$.
Since $\Di\BT(\Bv,\pi)=\mu(\De\Bv+\nabla\di\Bv)-\nabla\pi$,
we insert \eqref{150303_1} and \eqref{150303_2}-\eqref{150303_3} into \eqref{eq:3}
in order to obtain
\begin{align*}
&\pa_t\Bu-\mu\left(\De\Bu+\CA^\SST\nabla\di\Bu+\CA^\SST\nabla(\CB^\SST:\nabla\Bu)\right)+\CA^\SST\nabla\Fp \\
&=\mu\left(\sum_{i,j,k=1}^N\CV_{ijk}\left(\int_0^t\nabla\Bu(\xi,s)\intd s\right)\frac{\pa^2}{\pa\xi_j \pa\xi_k}\Bu \right. \\
&\left.+ \sum_{i,j,k=1}^N\left[\CW_{ijk}\left(\int_0^t\nabla\Bu(\xi,s)\intd s\right)
\int_0^t\nabla^2\Bu(\xi,s)\intd s\right] \frac{\pa }{\pa\xi_k}\Bu\right).
\end{align*}
Let $\CA^{-\SST}=(\CA^\SST)^{-1}=(\CA^{-1})^\SST$,
and multiply the last equation by $\CA^{-\SST}=(\BI+\BB)^\SST$ from the left-hand side. 
Thus,
\begin{align*}
&\pa_t\Bu -\mu(\De\Bu+\nabla\di\Bu)+\nabla\Fp
=- \BB^\SST\pa_t\Bu + \mu\BB^\SST\De\Bu + \mu\nabla(\CB^\SST:\nabla\Bu) \\
&\quad+\mu\left(\BI+\BB^\SST\right)
\sum_{i,j,k=1}^N\CV_{ijk}\left(\int_0^t\nabla\Bu(\xi,s)\intd s\right)\frac{\pa^2}{\pa\xi_j \pa \xi_k}\Bu \\
&\quad+\mu\left(\BI+\BB^\SST\right)
\sum_{i,j,k=1}^N\left[\CW_{ijk}\left(\int_0^t\nabla\Bu(\xi,s)\intd s\right)
\int_0^t\nabla^2\Bu(\xi,s)\intd s\right]\frac{\pa }{\pa\xi_k}\Bu,
\end{align*}
which completes Case \eqref{eq:11}.

{\bf Case \eqref{eq:12}.}
The equation \eqref{eq:12} follows from \eqref{eq:4} and the second relation of \eqref{150303_2} immediately.
This completes Case \eqref{eq:12}.

{\bf Case \eqref{eq:13}.}
Let $F(\xi)=\xi_N-d$, and then $\Ga=\{\xi\in\ws \mid F(\xi)=0\}$.
We can regard $\xi\in\Ga$ as $\xi=\xi(x,t)=\Te^{-1}(x,t)$ for $x\in \Ga(t)$,
where $\Te^{-1}(\cdot,t)$ is the inverse mapping of $\Te(\cdot,t)$.
Let $G(x,t)=F(\xi(x,t))$.
Since $\Ga(t)$ is defined by $G(x,t)=0$, 
we see that
the unit outward normal $\Bn$ to $\Ga(t)$ is given by
\begin{equation}\label{150304_1}
\Bn = \frac{\nabla G(x,t)}{|\nabla G(x,t)|}.
\end{equation}

Now it holds that
\begin{equation*}
\nabla G(x,t)
=\left(\sum_{j=1}^N\frac{\pa\xi_j}{\pa x_1}\frac{\pa F}{\pa \xi_j},\dots,
\sum_{j=1}^N\frac{\pa\xi_j}{\pa x_N}\frac{\pa F}{\pa\xi_j}\right)^\SST
=\CA^\SST\Be_N,
\end{equation*}
and thus we have by \eqref{150304_1} the relation:
\begin{equation}\label{150304_2}
\Bn = \frac{\CA^\SST\Be_N}{|\CA^\SST\Be_N|}.
\end{equation}
On the other hand, we see by \eqref{150303_2} that
\begin{equation*}
\BT(\Bv,\pi) = -\Fp\BI + \mu\left(\CA^\SST\nabla\Bu + (\nabla\Bu)^\SST\CA\right),
\end{equation*}
which, combined with \eqref{150304_2} and inserted into \eqref{eq:13}, furnishes 
\begin{equation*}
-\Fp\frac{\CA^T\Be_N}{|\CA^T\Be_N|}
+ \mu \left(\CA^T\nabla\Bu + (\nabla\Bu)^\SST\CA\right)\frac{\CA^\SST\Be_N}{|\CA^\SST\Be_N|} = 0.
\end{equation*}
The last equation is multiplied by $|\CA^\SST\Be_N|\CA^{-\SST}$ from the left-hand side to give
\begin{equation*}
-\Fp\Be_N + \mu \left(\nabla\Bu + \CA^{-\SST}(\nabla\Bu)^T\CA\right)\CA^\SST\Be_N = 0,
\end{equation*}
which, combined with $\CA^{-\SST}=\BI+\BB^T$ and $\CA=\BI+\CB$, implies 
\begin{equation*}
\BT(\Bu,\Fp)\Be_N = -\mu \left( \BD(\Bu)\CB^\SST 
+ \left\{(\nabla\Bu)^\SST\CB+\BB^\SST(\nabla\Bu)^\SST(\BI+\CB)\right\}(\BI+\CB^\SST)\right)\Be_N.
\end{equation*}
This completes Case \eqref{eq:13}.



\end{document}